\documentclass[11pt,a4paper]{amsart}
\usepackage{etex}
\usepackage[latin1]{inputenc}
\usepackage{amsmath}
\usepackage{amsthm}
\usepackage{amsrefs}
\usepackage{amsfonts}
\usepackage[all,cmtip]{xy}
\usepackage{amsfonts}
\usepackage{amssymb}
\usepackage{amscd}
\usepackage{bbm}
\usepackage{pictexwd,dcpic}
\usepackage[T1]{fontenc}
\usepackage{lmodern}
\usepackage{setspace}
\usepackage{stmaryrd}
\usepackage{mathpartir}
\def\C{\mathcal{C}}
\def\D{\mathcal{D}}

\def\S{\mathcal{S}}
\def\P{\mathcal{P}}
\def\E{\mathcal{E}}

\def\T{\mathbb{T}}
\def\Tp{\mathbb{T}'}
\def\tbar{t(\vec{\bar{x}},\vec{x})}
\def\sbar{s(\vec{\bar{x}},\vec{x})}

\def\lb{\lbrace}
\def\rb{\rbrace}
\newcommand\frakfamily{\usefont{U}{yfrak}{m}{n}}
\DeclareTextFontCommand{\textfrak}{\frakfamily}

\theoremstyle{plain}\newtheorem{lemma}{Lemma}[section]
\theoremstyle{plain}\newtheorem{cor}[lemma]{Corollary}
\theoremstyle{plain}\newtheorem{theo}[lemma]{Theorem}
\theoremstyle{definition}
\theoremstyle{definition}\newtheorem{defin}[lemma]{Definition}
\theoremstyle{remark}\newtheorem{exam}[lemma]{Example}
\theoremstyle{remark}
\theoremstyle{definition}
\theoremstyle{definition}
\theoremstyle{definition}\newtheorem{cons}{Construction}
\theoremstyle{definition}
\theoremstyle{plain}\newtheorem{prop}[lemma]{Proposition}

\author{Dimitris Tsementzis}
\title{A syntactic characterization of Morita equivalence}
\begin{document}

\date{\today}

\address{Department of Philosophy, Princeton University}
\curraddr{1879 Hall, Princeton, NJ 08544, USA}
\email{dtsement@princeton.edu}

\maketitle

\begin{abstract}
We characterize Morita equivalence of theories in the sense of Johnstone in \cite{Elephant} in terms of a new syntactic notion of a common definitional extension developed by Barrett and Halvorson in \cite{Barrett} for cartesian, regular, coherent, geometric and first-order theories. This provides a purely syntactic characterization of the relation between two theories that have equivalent categories of models naturally in any Grothendieck topos.
\end{abstract}

\tableofcontents

\section*{Introduction}


In \cite{Elephant} Johnstone defines a signature-independent notion of equivalence between first-order theories, which he calls Morita equivalence (and which we will call J-Morita equivalence). 
With altogether different motivations in mind, Barrett and Halvorson in \cite{Barrett} have defined a purely syntactic notion to relate (for their purposes \emph{classical, first-order}) theories which they also call Morita equivalence (and which we will call T-Morita equivalence).
Roughly speaking, two theories are T-Morita equivalent if they have a common definitional extension (\cite{Hodges}, Chapter 2.6) where in addition to defining new function and relation symbols in terms of available formulas one is also allowed to define four new types of sort symbols: product sorts, coproduct sorts, subsorts and quotient sorts.

In this paper we prove an equivalence between the topos-theoretic notion of J-Morita equivalence and (appropriate generalizations of) the new syntactic notion of T-Morita equivalence.
As such we answer the following question: 
Suppose you are given an equivalence between the categories of models (natural in any cocomplete topos) of two (regular, cartesian, coherent or geometric) theories $\T$ and $\Tp$. Then from a purely syntactic point of view, how are $\T$ and $\Tp$ related? Our investigation may thus be seen as an inversion of the kind of investigation undertaken e.g. in \cite{AwoFors} and \cite{MakkaiSD, MakkaiDD}. There the question is asked: what extra structure do we need to impose on the category of models of a theory in order to recover the theory up to up to equivalence of its syntactic category? Here we ask: if we identify a theory with its category of models, then what can we recover it \emph{up to}? 

Our work here is also inspired by (and may be seen as a ``logical'' version of) I. Moerdijk's work \cite{Moerdijk1, Moerdijk2, Moerdijk3, Moerdijk4} on the representation theory of Grothendieck toposes (cf. \cite{Moerdijk5} for an overview).
In addition to the projects pursued by Barrett and Halvorson  \cite{Barrett, Halv1, Halv2} our work also has obvious connections with the research programme developed by Caramello \cite{Caram, Caram2012, Caram2012b, Caram2009} based around the idea of exploiting J-Morita equivalences between geometric theories to transfer mathematical results from one theory to another. More classical investigations along these lines can be found in \cite{Pinter, Mycielski} and connections with Shelah's work are explained nicely in \cite{Harnik}.
Furthermore, our syntactic characterization is very type-theoretic in flavour since the main innovation of the new syntactic notion of Morita equivalence is to allow one to define new sort symbols from old ones in much the same way that in type theories one defines new types from old ones, 
though we restrict ourselves here only to (fragments of) first-order logic which correspond more closely to what have been called ``logic-enriched type theories'' (e.g. \cite{PAT, Maietti, MaiSam, MoePal, GamAcz}.)
But none of them have been used, to the best of our knowledge, as tools for studying equivalences of classifying toposes of theories, which is our main concern here.
Nevertheless the technical connections are obvious and we plan to explore the type-theoretic versions of our results in future work.

\vspace{0.3cm}

\emph{Outline of the Paper.} In Section \ref{prelim}, we quickly go through some preliminaries, fixing notation.
In Section \ref{tj} we define the two above-described notions of Morita equivalence for coherent theories and then go on in Sections \ref{tmortojmor} and \ref{jmortotmor} to prove that these notions coincide in the coherent fragment of first-order logic (Theorem \ref{TiffJ}.) 
We then go on in Section \ref{gen} to consider how Theorem \ref{TiffJ} translates to other fragments of first-order logic as well as to full first-order logic.
In this latter case we prove Theorem \ref{morlTiffJ} which says that two first-order theories are T-Morita equivalent if and only if their Morleyizations are J-Morita equivalent. 
Finally, in Section \ref{appimp} we discuss some mathematical applications and sketch future directions of research.

\section{Preliminaries}\label{prelim}

We assume familiarity with first-order categorical logic as can be found e.g. in \cite{MakkaiReyes} and D1 of \cite{Elephant}. We also assume some familiarity with the theory of classifying toposes as described e.g. in \cite{MakkaiReyes}, \cite{Elephant} or Chapter VIII of \cite{MM}. Throughout this paper our terminology and notation will follow \cite{Elephant} and our deductive system can be taken to be exactly the one presented there. We will here review a few standard concepts and make explicit some notational conventions that will be of use to us in what follows.

Since we will mainly be dealing with fragments of first-order logic which do not necessarily include implication, negation and universal quantification we will employ turnstile ($\vdash$) notation for the sequents we consider. So a given sequent $\sigma$ over a signature $\Sigma$ will be written as
$\phi \vdash_{\vec{x}} \psi$
where $\phi, \psi$ are $\Sigma$-formulas and $\vec{x}$ is a context of variables appropriate to both of them. We will often omit explicit mention of the variables binding our sequent when it is clear from the context. For a given sequent $\sigma$ over a signature $\Sigma$ and theory $\T$ also over $\Sigma$ we write
$\T \models \sigma$
to indicate that the sequent is derivable from (the sequents defining) $\T$. As usual, a theory $\T$ is identified with a set of sequents (the deductive closure of the axioms relative to our deductive system.) We will also sometimes separate sequents by commas if we want to list several sequents that a theory contains, e.g. $\T \models \sigma_1 , \dots, \sigma_n$ and we will also use the symbol $\equiv$ where necessary to indicate that two sequents  or formulas are (grammatically) identical.
We will usually abbreviate the variable contexts of sequents whenever these are clear from the context.
Similarly we will often write a single formula for a conjunction of formulas depending on variables ranging over some index. 
Furthermore whenever we are considering two signatures $\Sigma_1, \Sigma_2$ with $\Sigma_2 \supset \Sigma_1$ we will will write $\Sigma_2 \setminus \Sigma_1 \textbf{-Sort}$ for $\Sigma_2\textbf{-Sort} \setminus \Sigma_1 \textbf{-Sort}$ and similarly for $\textbf{Fun}$ and $\textbf{Rel}$.
We will also use the notation $\sigma_S$ to denote the set of sequents defining a symbol $S \in \Sigma_2 \setminus \Sigma_1$. 

Since our system contains the cut rule we will use the following convention: Given a theory $\T$ over a signature $\Sigma$, and $\Sigma$-formulas $\phi_1, \phi_2, ... \phi_n$ such that
$
\T \models \phi_1 \dashv \vdash \phi_2,
 \phi_2 \dashv \vdash \phi_3,
\dots,
 \phi_{n-1} \dashv \vdash \phi_n
$
we will write
$
\T \models \phi_1 \dashv \vdash \phi_2 
 \dashv \vdash \phi_3 
\dots  
 \dashv \vdash \phi_n
$
to indicate the derivation by $\T$ of (the two sequents represented by) $\phi_1 \dashv \vdash \phi_n$, i.e. that $\T \models \phi_1 \vdash \phi_n, \phi_n \vdash \phi_1$. When we do so we will usually justify each step one formula at a time. 

We will mainly be concerned with the \emph{coherent} fragment of first-order logic and our main results will be proved in detail for (constructive) coherent logic.
When we consider other fragments of first-order logic in Section \ref{gen} then the deductive system we have in mind will contain only the rules for the relevant connectives.


Secondly, our system includes the following two rules as axioms
\begin{align*}
\phi \wedge (\psi \vee  \chi) &\dashv \vdash_{\vec{x}} (\phi \wedge \psi) \vee (\phi \wedge \chi) \label{eq:dist} \tag{\text{Dist}} \\
\phi \wedge (\exists y \psi) &\dashv \vdash_{\vec{x}} \exists y(\phi \wedge \psi) \label{eq:frob} \tag{\text{Frob}}
\end{align*}
where in (\ref{eq:frob}) $y$ does not appear among the $\vec{x}$. 
We will also have occasion to refer to the following derived rules (whenever they make sense):
\[
x=z \dashv \vdash_{x,z \colon S} \exists y \colon S (x=y \wedge y=z) \label{etran} \tag{$\exists$-tran}
\]
\[
\exists \vec{x} (\phi \vee \psi) \dashv \vdash_{\vec{x}} \exists \vec{x} \phi \vee \exists \vec{x} \psi \label{evee} \tag{$\exists \vee$}
\]
where $\phi$ and $\psi$ are formulas with (some) free variables among the $\vec{x}$. 

%

Finally, recall 
that given a theory $\T$ over a signature $\Sigma$ we define the syntactic category $\C_\T$ of $\T$ to be the category whose objects are $\Sigma$-formulas-in-context $\vec{x}.\phi$ up to renaming of variables
 and whose arrows are $\T$-provable equivalence classes of $\T$-provably functional relations between such formulas. More precisely, a morphism $[\theta] \colon \lb \vec{x}. \phi \rb \rightarrow \lb \vec{y}. \psi \rb$ is a $\Sigma$-formula $\theta$ in the context $\vec{x},\vec{y}$ such that:
\[
\T \models \theta \vdash_{\vec{x},\vec{y}} \phi \wedge \psi,  \theta \wedge \theta[\vec{z}/\vec{y}] \vdash_{\vec{x},\vec{y},\vec{z}} \vec{z} = \vec{y},
\phi \vdash_{\vec{x}} \exists \vec{y} \theta(\vec{x}, \vec{y})
\]
Depending on the logical complexity of the theory $\T$ the syntactic category $\C_\T$ carries the appropriate logical structure. This if $\T$ is coherent (resp. cartesian, regular, geometric) then $\C_\T$ is a coherent (resp. cartesian, regular, geometric) category (cf. \cite{Elephant}, D1.4.10 and D1.4.2). With this in mind let us record the following well-known lemmas that we will often refer to in the rest of the paper:


\begin{lemma}[\cite{Elephant}, D1.4.4(iv)]
Every subobject of an object $\lb \vec{x}. \phi \rb$ in $\C_\T$ is isomorphic to one of the form 
\[
\lb \vec{x} . \psi \rb \overset{[\psi]}{\hookrightarrow} \lb \vec{x}. \phi \rb
\]
where $\T \models \phi \vdash_{\vec{x}} \psi$. Furthermore $\lb \vec{x}. \psi \rb \leq \lb \vec{x}. \chi \rb$ as subobjects of $\lb \vec{x}. \phi \rb$ if and only if $\T \models \psi \vdash_{\vec{x}} \chi$.
\end{lemma}

\begin{lemma}[\cite{Elephant}, D1.4.10]
\begin{enumerate}
\item If $\T$ is at least a regular theory, then given a morphism $[\theta] \colon \lb \vec{x}. \phi \rb \rightarrow \lb \vec{y}. \psi \rb$ in $\C_\T$ its image is given by the subobject $\lb \vec{y}. \exists \vec{x} \theta \rb$ of $\lb \vec{y}. \psi \rb$.
\item If $\T$ is at least a regular theory, then a morphism $[\theta] \colon \lb \vec{x}. \phi \rb \rightarrow \lb \vec{y}. \psi \rb$ in $\C_\T$ is a regular epi if and only if $\T \models \psi \vdash_{\vec{y}} \exists \vec{x} \theta$.
\item If $\T$ is at least a coherent category, and $\lb \vec{x}. \psi \rb$, $\lb \vec{x}. \chi \rb$ are two subobjects of $\lb \vec{x}. \phi \rb$ then their sup is given by the subobject $\lb \vec{x}. \psi \vee \chi \rb$.
\end{enumerate}
\end{lemma}

\section{T-Morita and J-Morita}\label{tj}

We now define the two different notions of Morita equivalence that we will go on to relate. 
In this and the next section we will restrict ourselves only to coherent theories. 


\begin{defin}
We call two coherent theories $\T$ and $\Tp$ \textbf{J-Morita equivalent} if they have equivalent classifying toposes.
\end{defin}





Now let $\T_1$ be a coherent theory over a signature $\Sigma_1$ and let $\T_2$ be a theory over a signature $\Sigma_2 \supset \Sigma_1$. 
We say that $\T_2$ is a \emph{definitional extension} (cf. \cite{Hodges}) of $\T_1$ if every symbol in $\Sigma_2 \setminus \Sigma_1$ is explicitly defined in $\T_1$ by a coherent $\Sigma_1$-formula. There are also admissibility conditions when defining new function or relation symbols (cf. \cite{Hodges} or \cite{Barrett}). 
Following \cite{Barrett} we define the following four sets of sequents, expressing the definability of certain new sorts in the expanded signature $\Sigma_2$ in terms of $\Sigma_1$. 

%

We say that $S_1 \times S_2 \times \dots \times S_n \in \Sigma_2\textbf{-Sort}$ (which we also write as $\prod_{i=1}^n S_i$) is a \emph{product sort} of $S_1, S_2, ... , S_n \in \Sigma\textbf{-Sort}$ with projections $\pi_i \colon S_1 \times... \times S_n \rightarrow S_i \in \Sigma_2\textbf{-Sort}$ for $i = 1,\dots,n$ if $\T_2$ contains the following sequents:
\[
\top \vdash_{x_i \colon S_i} \exists x \colon \prod_{i=1}^n S_i (\pi_1 (x) = x_1 \wedge ... \wedge \pi_n (x) = x_n) \label{eq:prod1} \tag{\text{prod$_1$}}
\]
\[
(\bigwedge_{i=1}^n \pi_i (x) = x_i) \wedge (\bigwedge_{i=1}^n \pi_i (z) = x_i ) \vdash_{x_1,...,x_n,x,z} x=z \label{eq:prod2} \tag{\text{prod$_2$}}
\]

We say that $S_1 \amalg S_2 \dots \amalg S_n  \in \Sigma_2\textbf{-Sort}$ is a \emph{coproduct sort} of $S_1, S_2, \dots, S_n \in  \Sigma_1\textbf{-Sort}$ with injections $\rho_i \colon S_i \rightarrow S_1 \amalg \dots \amalg S_n \in  \Sigma_2\textbf{-Fun}$ for $i=1,\dots,n$ 
if $\T_2$ contains the following sequents:
\[
\top \vdash_{x \colon \coprod_{i=1}^n S_i} \bigvee_{i=1}^n \exists x_i \colon S_i (\rho_i (x_i) = x) \label{eq:cop1} \tag{\text{cop$_1$}}
\]
\[
\rho_i (x_i) = x \wedge \rho_i (x'_i) = x \vdash_{x_i,x'_i,x} x_i = x'_i \text{ for all $i=1,...,n$}\label{eq:cop2} \tag{\text{cop$_2$-$i$}}
\]
\[
\rho_i (x_i) = x \wedge \rho_j (x_j) = x \vdash_{x_i \colon S_i,x_j \colon S_j x} \bot \text{ for all $i \neq j \in \lb 1,...,m \rb$}  \label{eq:cop3} \tag{\text{cop$_3$}}
\]
We say that $S \in \Sigma_2\textbf{-Sort}$ is a \emph{subsort} of a sort $T \in \Sigma_1\textbf{-Sort}$ defined by a coherent $\Sigma_1$-formula $\phi$ and a function symbol $i \colon S \rightarrow T \in \Sigma_2\textbf{-Sort}$ if $\T_2$ contains the following sequents:
\[
\phi(x) \dashv \vdash_{x \colon T} \exists y \colon S (i(y)=x) \label{eq:sub1} \tag{\text{sub$_1$}}
\] 
\[
i(x)=i(y) \vdash_{x,y\colon S} x=y \label{eq:sub2} \tag{\text{sub$_2$}}
\]
We say that $S \in \Sigma_2\textbf{-Sort}$ is a \emph{quotient sort} of $T \in \Sigma_1\textbf{-Sort}$ defined by a  $\T_1$-provable equivalence relation $\phi$ and a function symbol $\epsilon \colon T \rightarrow S \in \Sigma_2\textbf{-Sort}$ if $\T_2$ contains the following sequents:
\[
\epsilon (x) = \epsilon (y) \dashv \vdash_{x,y \colon T} \phi (x,y) \label{eq:quot1} \tag{\text{quot$_1$}}
\]
\[
\top \vdash_{x \colon S} \exists y \colon T (\epsilon (y) = x) \label{eq:quot2} \tag{\text{quot$_2$}}
\]

It is important to note that if $S$ is defined as a subsort of $T$, then $\phi$ is allowed to be $\bot$, i.e. we allow ourselves to define an ``empty'' subsort. (This automatically commits to a semantics that permits the interpretation of sorts as empty sets or initial objects.)
We also make the following restriction: new coproduct, product, subsorts and quotient sorts can only be defined on sorts $S_i$ that our base theory proves are non-empty, i.e. 
$\T \models \top \vdash \exists x_i \colon S_i \top$. 
(Whether a subsort is empty or not, of course, is a substantial question, since it implies that a certain sequent is derivable.) 
For the sake of Section \ref{jmortotmor} we must make one further very minor restriction: all the theories $\T$ we consider will be such that their signature contains at least one sort symbol $S$ that $\T$ proves is inhabited, i.e. such that $\T \models \top \vdash \exists x \colon S (\top)$. We require this in order always to be able to define a ``singleton sort'' as will be explained in more detail there.

\begin{defin}[Morita extension, \cite{Barrett}] \label{morex}
Let $\T_1$ and $\T_2$ be coherent theories over signatures $\Sigma_1$ and $\Sigma_2$ respectively. We say that $\T_2$ is a \textbf{Morita extension} of $\T_1$ if $\T_2$ is a definitional extension of $\T_1$ and every symbol in $\Sigma_2 \setminus \Sigma_1 \textbf{-Sort}$ is either a product sort, a coproduct sort, a quotient sort or a subsort and $\T_2$ contains in each case the appropriate sequents as defined above.
\end{defin}


The following proposition is immediate from the definitions.

\begin{prop}\label{cemor}
If $\T_2$ is a Morita extension of $\T_1$ then $\T_2$ is a conservative extension of $\T_1$.
\end{prop}

Given this definition of a Morita extension we may say that a \emph{(finite) Morita span} between two coherent theories $\T$ and $\T'$ is given by (finite) collections of theories
$\T=\T_0, \T_1,..., \T_n$ and $\Tp=\Tp_0, \Tp_1,..., \Tp_m$ such that for each $i=1,...,n$ each $\T_i$ is a Morita extension of $\T_{i-1}$ and for each $j=1,...,m$ each $\T_j$ is a Morita extension of $\T_{j-1}$ and also that $\T_n = \T_m$ where $(=)$ is to be understood as \emph{logical equivalence}: each of the axioms of one theory are derivable from the axioms of the other. We will also call a (finite) series $\T=\T_0, \T_1,..., \T_n$ of theories all of which are successively Morita extensions of each other a \emph{(finite) Morita chain} from $\T$ to $\T_n$.
We call two theories $\T$ and $\Tp$ \emph{pre-T-Morita equivalent} if they can be connected by a Morita span -- this is essentially the notion found in \cite{Barrett}.
The reason we don't define T-Morita equivalence as pre-T-Morita equivalence is because the latter fails to be transitive (as a relation between theories). 
To fix this we take the transitive closure of pre-T-Morita equivalence as our notion of T-Morita equivalence. 

\begin{defin}\label{Tmor}
We call two theories $\T$, $\Tp$ (over signatures $\Sigma$, $\Sigma'$) \textbf{T-Morita equivalent} if and only if there is a theory $\T''$ over a signature $\Sigma''$ such that $\Sigma \cap \Sigma'' = \varnothing$ and $\Sigma''$ differs from $\Sigma''$ only in renaming the symbols in $\Sigma \cap \Sigma'$ (keeping their arities fixed in the case of function and relation symbols) and $\T''$ differs from $\T$ only in replacing the symbols in $\Sigma'$ with symbols in $\Sigma''$ and $\T''$ is pre-T-Morita equivalent to $\T$.
\end{defin}


The following then follows easily.

\begin{prop}\label{Tmorer}
T-Morita equivalence is an equivalence relation on the class of (coherent) theories.
\end{prop}


\section{T-Morita to J-Morita}\label{tmortojmor}


Let us outline the general strategy of the proof. Let $\T_1$ and $\T_2$ be coherent theories. The coherent Grothendieck topology on their syntactic categories gives us sites $(\C_{\T_1},J_1)$ and $(\C_{\T_2},J_2)$ and the topos of sheaves on those sites gives us the classifying toposes of $\T_1$ and $\T_2$. So what we care about is whether these two sites give rise to equivalent categories of sheaves.
A version of Verdier's ``Comparison Lemma'' from SGA4 which provides a sufficient condition for this to be the case and we will prove that this condition is satisfied on the assumption that $\T_2$ is a Morita extension of $\T_1$. 


Let $(\C,J)$ be any site and let $i \colon \D \hookrightarrow \C$ be a full and faithful functor. Then there is a topology $J_{\D}$ on $\D$ which we call the \emph{induced topology} defined for every $A$ in $\D$ by
$
J_\D (A) = J(A) \cap \text{Mor}(\D)
$. We call $(\D,J_{\D})$ the \emph{induced site}.
There is an induced functor $i^* \colon \textbf{Sh}(\C,J) \rightarrow \textbf{Sh}(\D,J_\D)$ defined by precomposition in the usual way. 
We have:

\begin{lemma}[Comparison Lemma]
Let $(\C,J)$ be a site and let $i \colon \D \hookrightarrow \C$ be a full and faithful functor and let $(\D,J_{\D})$ be the induced site. If every object $A$ of $\C$ has a covering sieve $R \in J(A)$ generated by arrows all of whose domains are in $\D$, then $i^*$ is an equivalence.
\end{lemma}

We record the following fact about induced topologies that we shall require below.

\begin{lemma}[\cite{Elephant}, C.2.2.2(i)]\label{elephlem}
Let $(\C,J)$ be a site and let $\D$ be a full subcategory of $\C$. Then a sieve $S$ on an object $A$ of $\D$ is $J_\D$-covering if and only if the sieve $\bar{S}$ in $\C$ generated by the members of $S$ is $J$-covering.
\end{lemma}

So now assume $\T_2$ is a Morita extension of $\T_1$ and let $\Sigma_2 \supset \Sigma_1$ be their respective signatures. Assuming Lemma \ref{key} - which we prove below - we have the following:

\begin{lemma}\label{fullfaith}
There is a full and faithful embedding $i \colon \C_{\T_1} \hookrightarrow \C_{\T_2}$.
\end{lemma}
\begin{proof}
We have that $\T_2$ is a Morita extension of $\T_1$. Then we have that 
$
\T_2 = \T_1 \cup \lbrace  \sigma_S \vert S \in \Sigma_2 \setminus \Sigma_1 \rbrace
$
where $\sigma_S$ are sequents defining the symbols in $\Sigma_2 \setminus \Sigma_1$ in terms of $\T_1$-formulas. Clearly we get an inclusion functor 
$
i \colon \C_{\T_1} \rightarrow \C_{\T_2} 
$
that takes
$
\lb \vec{x}.\phi \rb \mapsto \lb \vec{x}.\phi \rb
$
since $\phi$ will be a $\Sigma_1$-formula and therefore automatically a $\Sigma_2$-formula since $\Sigma_2 \supset \Sigma_1$. On arrows we define $i$ in the obvious way as $i([\theta]) = [\theta]$. This functor is clearly faithful since $\T_2$ is a conservative extension of $\T_1$ (Proposition \ref{cemor}) and therefore $\T_1$-provable equivalence classes of $\Sigma_1$-formulas are the same as $\T_2$-provable equivalence classes of $\Sigma_1$-formulas. To see that it is full, suppose $[\theta] \colon \lb \vec{x}.\phi \rb \rightarrow \lb \vec{y}.\psi \rb$ is an arrow in $\C_{\T_2}$ where $\phi$ and $\psi$ are $\Sigma_1$-formulas and the variables in $\vec{x}$ and $\vec{y}$ are all of sorts in $\Sigma_1$ (i.e. both objects are in the image of $i$.) Since $\theta$ has free variables only of sorts in $\Sigma_1$ by Lemma \ref{key} below there is a $\Sigma_1$-formula $\theta^*$ (with free variables $\vec{x},\vec{y}$) such that $\T_2 \models \theta \dashv \vdash_{\vec{x},\vec{y}} \theta^*$ and by construction of the syntactic category this means that $i([\theta^*])=[\theta^*]=[\theta]$. 
\end{proof}



\begin{lemma}\label{indeqcoh}
The topology induced by $i$ as in Lemma \ref{fullfaith} on $\C_{\T_1}$  when $\C=(\C_{\T_2},J_2)$ where $J_2$ is the coherent topology on $\C_{\T_2}$ coincides with the coherent topology on $\C_{\T_1}$.
\end{lemma}
\begin{proof}
Write $K$ for topology on $\C_{\T_1}$ induced by the coherent topology $J_2$ on $\C_{\T_2}$. Without loss of generality we will consider only basic covers. We have
$
\lb [\theta_i] \colon \lb \vec{x}_i . \phi_i \rb \rightarrow \lb \vec{y} . \psi \rb \rb \in J_1 (\lb \vec{y}.\psi \rb)  \text{ iff } \T_1 \models \bigvee_i \exists\vec{x}_i \theta_i \dashv \vdash \psi 
$
\text{ iff }
$
\T_2 \models \bigvee_i \exists\vec{x}_i \theta_i \dashv \vdash \psi 
$
\text{ iff } 
$
\lb [\theta_i] \rb \in J_2 (\lb \vec{y}.\psi \rb) 
$
\text{ iff } 
$
\lb [\theta_i] \rb \in K (\lb \vec{y}.\psi \rb)
$
where the first biconditional is simply the definition of the coherent topology, the second one follows from the fact that a $\T_2$ is a conservative extension of $\T_1$, the third one is again by definition and the fourth one follows from Lemma \ref{elephlem}.
\end{proof}

So Lemmas \ref{fullfaith} and \ref{indeqcoh} gives us the first two conditions in the statement of the Comparison Lemma. The next series of lemmas aim to establish the remaining condition. First we require some definitions about how to relate variables of ``new'' sorts to variables of ``old'' sorts. We follow \cite{Barrett} in calling these sequents \emph{codes}.
Let $x_1,...,x_n$ be variables of sorts in $\Sigma_2 \setminus \Sigma_1\textbf{-Sort}$. We say that a \emph{code} for $x_1,...,x_n$ is a $\Sigma_2$-formula 
\[
\xi(x_1,...,x_n,y_1,y_{11},...,y_{1m_1},...,y_n,y_{n1},...,y_{nm_n}) = \bigwedge_{i=1}^n \xi_i(x_i,y_i,y_{i1},y_{im_i})
\]
where each conjunct $\xi_i$ depends on the type of variable it codes.
If $x_i$ is of product sort $S_1 \times ... \times S_{m_i}$ with projections $\pi_1,...,\pi_{m_i}$ then
\[
\xi_i(x_i,y_i,y_{i1},...,y_{im_i}) \equiv \bigwedge_{k=1}^{m_i} \pi_k (x_i) = y_{ik}
\]
If $x_i$ is of coproduct sort $S_1 \amalg ... \amalg S_{m_i}$ with injections $\rho_1,...,\rho_{m_i}$ then
$
\xi^k_i(x_i,y_i,y_{i1},...,y_{im_i}) \equiv \rho_k (x_i) = y_{ik}
$
for each $k=1,...,m_i$. This means that there are $k$ choices for a valid code, and we make one.
If $x_i$ is of subsort $S \subset T$ with injection $i \colon S \hookrightarrow T$ then
$
\xi_i(x_i,y_i,y_{i1},...,y_{im_i}) \equiv i(x_i) = y_i
$.
If $x_i$ is of quotient sort $T= S/{\sim}$ with projection $\epsilon \colon S \rightarrow T$ then
$
\xi_i(x_i,y_i,y_{i1},...,y_{im_i}) \equiv \epsilon(y_i) = x_i
$.
We also stipulate that $\top$ is a code for the empty context of variables, i.e. the \emph{empty code}.
Finally it is important to note that a code for a variable of coproduct sort is not a disjunction of all possible codes. This proves crucial in the proof of Lemma \ref{codefunc} where we prove that codes behave in a ``provably functional'' way as made precise there.

\begin{lemma}\label{codex}
For any code $\xi (\vec{x},\vec{y})$ for any variables $\vec{x}$ of sorts in $\Sigma_2 \setminus \Sigma_1\textbf{-\emph{Sort}}$ and variables $\vec{y}$ of sorts in $\Sigma_1$ we have
\[
\T_2 \models \top \vdash_{\vec{x}} \bigvee_k \exists \vec{y}_k \xi^k(\vec{x},\vec{y}_k)
\]
where the index of the disjunction is taken over all possible codes for a variable of a coproduct sort if there are any among the $\vec{x}$. If there are no such variables then the disjunction symbol may be ignored.
Similarly, for any code other than one that contains a subsort code we have
\[
\T_2 \models \top \vdash_{\vec{y}} \exists \vec{x} \xi(\vec{x},\vec{y})
\]
\end{lemma}

The problem with subsort codes $S' \subset S$ is that it won't necessarily be true that for any variable $y$ of sort $S$ there is a variable $x \colon S'$ such that $i(x)=y$ because the defining formula $\phi$ of the subsort might not be true of all $y \colon S$. 
This is a subtlety that proves important in the proofs of Lemma \ref{prekey} and \ref{key}. The solution, roughly, is this: whenever we want to ``interpret'' a $\Sigma_2$-formula containing a variable $x$ of some subsort as a $\Sigma_1$-formula, we send every instance of $x$ to $y \wedge \phi(y)$, i.e. to a variable of the original sort that ``is in the image'' of $i$.

\begin{proof}[Proof of Lemma \ref{codex}]
It suffices to prove the statement for the basic cases of codes, since any conjunctions of those will also clearly satisfy the conclusions. All these basic cases follow immediately. In the case of subsorts for example we need to show that
$
\T_2 \models \top \vdash_{x \colon S \subset T} \exists y \colon T (i(x)=y)
$
and this follows from the fact that $i$ is a function symbol and $T$ is assumed non-empty. 
The three cases in the other direction follow similarly. Let us do only the coproduct case in order to make it clear why a disjunction is no longer necessary. Given a variable $y_k \colon S_k$ for $k \in \lb 1,...,n \rb$ we know that there is $x \colon S_1 \amalg ... \amalg S_n$ such that $\rho_k (y_k) = x$, i.e. that
$
\T_2 \models \top \vdash_{y_k \colon S_k} \exists x \colon S_1 \amalg ... \amalg S_n (\rho_k(y_k) = x)  
$.
But the formula in the scope of the existential quantifier is exactly $\xi^k (x,y_k)$ which means that 
$
\T_2 \models \top \vdash_{y_k \colon S_k} \exists x \colon S_1 \amalg ... \amalg S_n \xi^k (x,y_k)
$
as required.
\end{proof}

\begin{lemma}[``Functionality of codes'']\label{codefunc}
Let $\vec{x}$ be a context of variables of sorts in $\Sigma_2 \setminus \Sigma_1$ and let $\vec{y}$ be variables of sorts in $\Sigma_1$. Let $\xi(\vec{x},\vec{y})$ be a code for $\vec{x}$. Let $\vec{z}$ be a distinct context of variables of the same sort and length as $\vec{x}$. Then we have
\[
\T_2 \models \xi(\vec{x},\vec{y}) \wedge \xi (\vec{z},\vec{y}) \vdash \vec{x} = \vec{z}
\]
\end{lemma}
\begin{proof}
Once again it suffices to prove the statement for the basic cases, which are more or less immediate.
In the case of subsorts for example let $\vec{x},\vec{z} = x,z \colon S \subset T$ where $S$ is a subsort of $T$ defined by $\phi$ and with injection $i$. Then we have
$
\T_2 \models i(x)=y \wedge i(z)=y \vdash z=x
$
since the left hand-side implies $i(x)=i(z)$ and then the right-hand side follows from (\ref{eq:sub2}).
\end{proof}
As its description suggests, Lemma \ref{codefunc} will be important in proving the second part of Lemma \ref{key} below where we have to show that conjuncts of codes and formulas define functional relations. The next two lemmas establish that every $\Sigma_2$-formula is $\T_2$-provably equivalent to a $\Sigma_1$-formula ``up to coding new variables into old ones.''
\begin{lemma}[``Recoding of terms'']\label{prekey}
Let $t(\vec{\bar{x}},\vec{x})$ be a $\Sigma_2$-term with $\vec{\bar{x}}$ variables of sorts in $\Sigma_1$ and $\vec{x}$ variables of sorts in $\Sigma_2 \setminus \Sigma_1$ and let $z$ be a variable of the same sort as $t$. Then we have
\[
\T_2 \models t(\vec{\bar{x}},\vec{x}) = z \dashv \vdash_{\vec{\bar{x}},\vec{x},z} \bigvee_j \exists \vec{y}_j (\xi_j(z,\vec{x},\vec{y}_j) \wedge \phi^j_t(z,\vec{\bar{x}},\vec{y}_j))
\]
where the $\xi_j$ are codes and each $\phi_t^j$ is a $\Sigma_1$-formula.
\end{lemma}
\begin{proof}
We take cases. 
The difficult cases are those when either $t$ contains variables of new sorts or $S$ is a new sort. The former follow easily but in the latter we need to consider each type of new sort separately. 
For example, if $t\equiv x_i \colon S$
and $S \in \Sigma_1\textbf{-Sort}$ we have
\[
\T_2 \models x_i = z \dashv \vdash_{x_i,z} \top \wedge x_i = z
\]
where $\phi_t \equiv x_i = z$ and we let $\xi = \top$, i.e. the empty code.
If $S \notin \Sigma_1\textbf{-Sort}$ then this case splits into four subcases, one for each new sort. 
If $S = \lb S_1 \times ... \times S_n, \pi_1,...,\pi_n \rb$ we have
\[
\T_2 \models x_i = z \dashv \vdash_{x_i,z} \exists y_1\exists y_2 \exists y_{i1} \exists y_{i2} (\pi_j(x_i) = y_j \wedge \pi_j (z) = y_{ij} \wedge y_j=y_{ij})
\]
which is of the required form with $\xi \equiv \pi_j(x_i) = y_j \wedge \pi_j (z) = y_{ij}$ and $\phi^j_t \equiv y_j=y_{ij}$.) The fact that $\T_2$ contains the above sequents follows immediately from (\ref{eq:prod1}).
The coproduct and quotient cases follow similarly, but there is an important subtlety in the case of subsorts.
%
%
Let $S= \lb S \subset T, i, \psi \rb$. We then have
\[
\T_2 \models x_i = z \dashv \vdash_{x_i,z} \exists y_1,y_2 \colon T (i(x_i)=y_1 \wedge i(z) = y_2 \wedge y_1 = y_2 \wedge \psi (y_1) \wedge \psi (y_2))
\]
which is of the required form with $\phi_t \equiv y_1 = y_2 \wedge \psi (y_1) \wedge \psi (y_2)$. The addition of $\psi (y_1) \wedge \psi (y_2)$ at the end of $\phi_t$ seems redundant here since in the sequent above we already have the conjuncts $i(x_i)=y_1 \wedge i(z) = y_2$ which by (\ref{eq:sub1}) are provably equivalent to the former, but the explicit form of $\phi_t$ as constructed here will be important in the proof of Lemma \ref{key} below where the presence of conjuncts such as $\psi(y_2)$ will prove essential. This is the subtlety that we referred to in the remark just after Lemma \ref{codex}. Now as to why the above sequent actually holds, we must invoke (\ref{eq:sub2}) as follows
\begin{align*}
\T_2 \models x_i = z &\dashv \vdash i(x_i) = i(z) \tag{\ref{eq:sub2}} \\
&\dashv \vdash \exists y_1,y_2 \colon T (i(x_i)=y_1  \tag{\ref{etran}} \\
&\! \! \!\wedge i(z) = y_2 \wedge y_1 = y_2 \wedge \psi (y_1) \wedge \psi (y_2))
\end{align*}

%
%

Now if $t \equiv f(t_1 (\vec{x},\vec{\bar{x}}),...,t_k(\vec{x},\vec{\bar{x}}))\colon S$ we proceed by induction, assuming the hypothesis holds for each term $t_i(\vec{x},\vec{\bar{x}})$, for $i=1,...,k$, i.e.
\[
t_i(\vec{x},\vec{\bar{x}}) = z_i \dashv \vdash \bigvee_j \exists \vec{y}_{ij} (\xi_{ij}(z_i,\vec{x},\vec{y}_{ij}) \wedge \phi^j_{t_i}(z_i,\vec{\bar{x}},\vec{y}_{ij})) \label{IH} \tag{IH}
\]
(We will refer to similar inductive hypotheses as (\ref{IH}) in the rest of this proof even when we drop the subscripts.)
If $f \in \Sigma_1\textbf{-Fun}$ we have
\begin{align*}
\T_2 \models t(\vec{x},\vec{\bar{x}}) = z &\dashv \vdash \bigwedge_{i=1}^k (\bigvee_j \exists \vec{y}_{ij} (\xi_{ij}(z_i,\vec{x},\vec{y}_{ij}) \wedge \phi^j_{t_i}(z_i,\vec{\bar{x}},\vec{y}_{ij}))) \\  
&\wedge f(z_1,...,z_k)=z \tag{\ref{IH}}\\ 
&\dashv \vdash \bigvee_j \exists \vec{y}_{1j} ... \exists \vec{y}_{kj} ((\bigwedge_{i=1}^k \xi_{ij}(z_i,\vec{x},\vec{y}_{ij}))\wedge((\bigwedge_{i=1}^k \phi_{t_i}^j (z_i,\vec{x},\vec{y}_{ij})) \\
& \wedge f(z_1,...,z_k)=z)) \tag{\ref{eq:dist}}
\end{align*}
which is of the required form since $f \in \Sigma_1\textbf{-Fun}$.
If $f \notin \Sigma_1\textbf{-Fun}$ there are two subcases. Either $f$ has arity with all sorts in $\Sigma_1$, in which case it is defined by $\T_2$, say by some formula $\psi(z_1,...,z_k,z)$. In that case we have the exact same sequents as above, except with $\psi(z_1,...,z_k,z)$ replacing $f(z_1,...,z_k)=z$. Otherwise, $f$ is one of $\pi_i, \rho_i, i$ or $\epsilon$ for some new sort. As before we need to take each subcase separately and they all follow straightforwardly except for subsorts. 
%
%
%
So take the case of a subsort injection $i$ defined by a formula $\psi$. We must once again apply a ``patch'' by adding an instance of $\psi$ where it appears not to be needed. Also note that we need only consider $i$ as applied to variables since there is no function symbol with codomain $S$ in $\Sigma_2$ (by assumption.) Thus from (\ref{etran}) and (\ref{eq:sub1}) we get
\[
\T_2 \models i(x) = z \dashv \vdash \exists z' \colon T (i(x) = z' \wedge z=z' \wedge \psi(z'))
\]
which is clearly of the required form since $\psi$ is a $\Sigma_1$-formula.
%
\end{proof}

\begin{lemma}[``Recoding of formulas'']\label{key}
Let $\psi(\vec{\bar{x}}, \vec{x})$ be a $\Sigma_2$-formula with $\vec{x}$ and $\vec{\bar{x}}$ as in Lemma \ref{prekey}. Then
\[
\T_2 \models \psi(\vec{\bar{x}}, \vec{x}) \dashv \vdash \bigvee_j \exists \vec{y}_j (\xi_j (\vec{x},\vec{y}_j) \wedge \psi^*_j (\vec{\bar{x}},\vec{y}_j))
\]
where each $\xi_j$ is a code and each $\psi_j^*$ is a $\Sigma_1$-formula. In addition, each
$
\theta_j \equiv \xi_j (\vec{x},\vec{y}_j) \wedge \psi^*_j (\vec{\bar{x}},\vec{y}_j)
$
is a $\T_2$-provably functional relation from $\psi_j^*$ to $\psi$, i.e. defines a morphism
$
[\theta_j] \colon \lb \vec{\bar{x}},\vec{y}. \psi^*_j \rb \rightarrow \lb \vec{x}, \vec{\bar{x}}.\psi \rb
$
in $\C_{\T_2}$.
\end{lemma}

\begin{proof}
The first part involves a long induction on complexity. 
If $\psi \equiv \tbar = \sbar \colon S$ and $S \in \Sigma_1\textbf{-Sort}$ then 
we get
\begin{align*}
\T_2 \models \tbar = \sbar &\dashv \vdash \exists w \colon S (\tbar = w \wedge \sbar = w)  \\ 
&\dashv \vdash \exists w((\bigvee_j \exists \vec{y}_{t,j} (\xi_{t,j} \wedge \phi_{t,j}))\wedge (\bigvee_k \exists \vec{y}_{s,k} (\xi_{s,k} \wedge \phi_{s,k}))) \tag{Lemma \ref{prekey}}  \\
&\dashv \vdash \exists w (\bigvee_{j,k} (\exists \vec{y}_{t,j}(\xi_{t,j} \wedge \phi_{t,j}) \wedge \exists \vec{y}_{s,k}(\xi_{s,k} \wedge \phi_{s,k})) \tag{\ref{eq:dist}} \\ 
&\dashv \vdash \exists w \bigvee_{j,k} \exists \vec{y}_{t,j} \exists \vec{y}_{s,k} (\xi_{t,j} \wedge \xi_{s,k} \wedge (\phi_{t,j} \wedge \psi_{s,k})) \tag{\ref{eq:frob}} \\ 
&\dashv \vdash \bigvee_{j,k} \exists \vec{y}_{t,j} \exists \vec{y}_{s,k} (\xi_{t,j} \wedge \xi_{s,k} \wedge \exists w:S(\phi_{t,j} \wedge \psi_{s,k})) \tag{\ref{evee}}
\end{align*}
where the first line of the deduction is simply (\ref{etran}).
If $S \notin \Sigma_1\textbf{-Sort}$ we must again consider all subcases.
If $S = \lb S_1 \times S_2, \pi_1,\pi_2 \rb$ then we have
\begin{align*}
\T_2 \models \tbar = \sbar &\dashv \vdash \exists w \colon S_1 \times \dots \times S_n (\tbar = w \wedge \sbar = w) \\ 
&\dashv \vdash \exists w ((\bigvee_j \exists \vec{y}_{t,j} \exists \vec{y}_w (\xi_{t,j} \wedge \phi_{t,j}))  \wedge (\bigvee_k \exists \vec{y}_{s,k} \exists \vec{y}_w (\xi_{s,k}  \wedge \phi_{s,k}))) \tag{\ref{prekey}}   \\ 
&\dashv \vdash \exists w(\bigvee_{j,k} \exists \vec{y}_{t,j} \exists \vec{y}_{s,k} \exists \vec{y}_w (\xi_{t,j} \wedge \xi_{s,k} \wedge \phi_{t,j} \wedge \phi_{s,j})) \tag{\ref{eq:dist}} \\ 
&\dashv \vdash \bigvee_{j,k} \exists \vec{y}_{t,j} \exists \vec{y}_{s,k} \exists \vec{y}_w ( \exists w(\xi_{t,j} \wedge \xi_{s,k}) \wedge \phi_{t,j} \wedge \phi_{s,j}) \tag{\ref{evee}} \\ 
&\dashv \vdash \bigvee_j \exists \vec{y}_{s,j} \exists \vec{y}_{t,j} \exists \vec{y}_w ((\widehat{\xi}_{t,j} \wedge \widehat{\xi}_{s,j}) \wedge \phi_{s,j} \wedge \phi_{t,j}) \tag{\ref{codex}} \\
&\dashv \vdash \bigvee_{j,k} \exists \vec{y}_{t,j} \exists \vec{y}_{s,k} ( \widehat{\xi}_{t,j} \wedge \widehat{\xi}_{s,k}) \wedge \exists \vec{y}_w (\phi_{t,j} \wedge \phi_{s,j})) \tag{\ref{eq:frob}}
\end{align*}
where $\vec{y}_w$ stands for the ($n$-tuple of) variables of sorts $S_1$ to $S_n$ coding the variable $w$ of sort $S_1 \times \dots \times S_n$ and $\widehat{\xi}$ stands for 
the same code as $\xi$ except we've removed the conjuncts coding $w$ through $\vec{y}_w$ (as above, we are allowed to do this because of Lemma \ref{codex}.)
Coproducts and quotients follow similarly but if $S= \lb S \subset T, i, \chi \rb$ then the patch that we mentioned in the proof of Lemma \ref{prekey} is going to be used essentially. We have
\begin{align*}
\T_2 \models \tbar = \sbar &\dashv \vdash \exists w \colon S (\tbar = w \wedge \sbar = w) \\
&\dashv \vdash \exists w \colon S ((\bigvee_j \exists \vec{y}_{t,j} \exists y_w (\xi_{t,j} (w, y_w, \vec{x},\vec{y}_{t,j}) \wedge \phi_{t,j}(\vec{\bar{x}},\vec{y}_{t,j},y_w)))) \\ 
\wedge &(\bigvee_j \exists \vec{y}_{s,j} \exists y_w (\xi_{s,j} (w, y_w, \vec{x},\vec{y}_{s,j}) \wedge \phi_{s,j}(\vec{\bar{x}},\vec{y}_{s,j},y_w))) \\ 
&\dashv \vdash \bigvee_j \exists \vec{y}_{t,j} \exists y_w (\exists w \xi_{t,j} (w, y_w, \vec{x},\vec{y}_{t,j}) \wedge \phi_{t,j}(\vec{\bar{x}},\vec{y}_{t,j},y_w)))) \\ 
\wedge &(\bigvee_j \exists \vec{y}_{s,j} \exists y_w (\exists w \xi_{s,j} (w, y_w, \vec{x},\vec{y}_{s,j}) \wedge \phi_{s,j}(\vec{\bar{x}},\vec{y}_{s,j},y_w))) \\
&\dashv \vdash \bigvee_j \exists \vec{y}_{s,j} \exists \vec{y}_{t,j} ((\xi_{t,j} (\vec{x},\vec{y}_{t,j}) \wedge \xi_{s,j}(\vec{x}, \vec{y}_{s,j})) \wedge (\exists y_w (\phi_{s,j} \wedge \phi_{t,j}))) \tag{$*$}
\end{align*}
which is of the required form. $\widehat{\xi}$ and $y_w$ are used in exactly the same way as in the product case above and every move up to the penultimate sequent is justified similarly. To move to the final sequent (labelled ($*$)) we invoke the fact that in the proof of Lemma \ref{prekey} we stipulated that any $\phi$ which contains a variable coding another variable belonging to a subsort will also contain a conjunct asserting that that variable satisfies the defining formula for the subsort. 
So in particular since each $\phi_{s,j}$ contains  a free variable $y_w$ of sort $T$ coding a variable of sort $S$, $\phi_{s,j}$ will also contain the conjunct $\chi(y_w)$. Since we know from the defining axioms of a subsort that
$\T_2 \models \exists w (i(w) = y_w) \dashv \vdash_{y_w \colon T} \chi (y_w) $
this allows us to safely move from $\xi$ to $\widehat{\xi}$. 
But since by (\ref{eq:sub1}) this is $\T_2$-provably equivalent to $\chi(y_w)$ and the latter is contained as a conjunct in both $\phi_{t,j}$ and $\phi_{s,j}$ we may drop it altogether. After we do so $y_w$ is no longer free in $\xi_t,j$ or $\xi_{s,j}$ and we may therefore push the quantifier in, which is what gives us the (right-hand side of) the final sequent. 

%
If $\psi \equiv R(t_1 (\vec{\bar{x}}, \vec{x}),...,t_k(\vec{\bar{x}}, \vec{x}))$ we once again have two cases. If $R \in \Sigma_1\textbf{-Rel}$ then for any $i=\lb 1,...,k \rb$ we have by Lemma \ref{prekey} that
$
\T_2 \models t_i (\vec{\bar{x}}, \vec{x}) = w_i \dashv \vdash \underset{j}{\bigvee} \exists \vec{y}_{ij} (\xi_{ij} \wedge \psi^*_{ij})
$. Using this fact we can define
$
\psi^* \equiv \exists w_1 ...  \exists w_k (\bigwedge_{i=1}^k \psi^*_{ij}) \wedge R(w_1,...,w_k)
$
which is a $\Sigma_1$-formula and 
$
\xi_j = \bigwedge_{i=1}^k \xi_{ij}
$
which is clearly a code. Then it follows easily that
$
\T_2 \models \psi \dashv \vdash \underset{j}{\bigvee} \exists \vec{y}_{1j} ... \exists \vec{y}_{kj} (\xi_j \wedge \psi^*)
$
which is of the required form.
If $R \notin \Sigma_2\textbf{-Rel}$ then this means that $R$ is definable in terms of a $\Sigma_1$-formula $\chi_R$, which we then use in exactly the same way as we used $R(w_1,...,w_k)$ above.

This completes the base case and we now move to the inductive step. Since we are only considering coherent formulas, we need only check the inductive step for conjunctions, disjunctions and existential quantifiers. Conjunctions and disjunctions follow easily. 
%
%
%
In the case of the existential quantifier let $\psi \equiv \exists x \colon S \phi \colon$ where $x \colon S$ is free in $\phi$ and $\phi$ satisfies the inductive hypothesis, i.e. $
\T_2 \models \phi \dashv \vdash_{x\colon S} \underset{j}{\bigvee} \exists\vec{y} (\xi_j \wedge \phi_j^*)
$. 
Now there are two subcases.
If $S \in \Sigma_1\textbf{-Sort}$ then $x$ does not appear in $\xi_j$ for any $j$ and so we immediately get that
\[
\T_2 \models \exists x \phi \dashv \vdash \exists x \bigvee_j \exists\vec{y} (\xi_j \wedge \phi_j^*) 
\dashv \vdash \bigvee_j \exists\vec{y} (\xi_j \wedge \exists x \phi_j^*)
\]
On the other hand, if $S \notin \Sigma_1\textbf{-Sort}$ we once again have four subcases. They all follow similarly. We do the coproduct case for variety.
%
%
%
So if $S = \lb S_1 \amalg \dots \amalg S_n, \rho_i \rb$ this means that $x \colon S_1\amalg  \dots \amalg S_n$ does not appear free in $\phi^*_j$ but appears free in $\xi_j$. Therefore $\exists x \xi_j$ contains a conjunct $\exists x (\rho_j (x_j) = x)$ for $j=\lb 1,\dots,n \rb$ and also the $x_j$ appear free in $\phi^*_j$. From Lemma \ref{codex} we have
\[
\top \dashv \vdash_{x_j \colon S_j} \exists x \bigvee_{j=1}^n (\rho_j (x_j) = x) \tag{$\dag$} 
\]
Thus we have
\begin{align*}
\T_2 \models \psi &\dashv \vdash \exists x  \colon S_1 \amalg \dots \amalg S_n \phi \dashv \vdash \exists x \bigvee_{j} \exists \vec{y}_j (\xi_j \wedge \phi^*_j) \\
&\dashv \vdash \bigvee_{j} \exists \vec{y}_j (\exists x \xi_j \wedge \phi^*_j) \tag{\ref{etran}}
\dashv \vdash \bigvee_j \exists \vec{y}_j ((\exists x (\rho_j(x_j) = x) \wedge \widehat{\xi}_j )\wedge \phi_j^*) \\ 
&\dashv \vdash \bigvee_j \exists \widehat{\vec{y}}_j (\widehat{\xi}_j \wedge \exists x_j\phi^*_j) \tag{$\dag$}
\end{align*}
where $\widehat{\vec{y}}_j $ means the variables in $\vec{y}_j$ except $x_j$.
%
%
%
So this completes the proof that every $\Sigma_2$-formula $\psi$ is $\T_2$-provably equivalent to a formula of the required form, i.e.
\[
\T_2 \models \psi \dashv \vdash \bigvee_j \exists \vec{y}_j (\xi_j (\vec{x},\vec{y}_j) \wedge \psi^*_j (\vec{\bar{x}},\vec{y}_j)) \tag{H} \label{H}
\]
Write $\theta_j$ for each $\xi_j (\vec{x},\vec{y}_j) \wedge \psi^*_j (\vec{\bar{x}},\vec{y}_j)$ as above. We must now show that each such $\theta_j$ defines a morphism in $\C_{\T_2}$, i.e. that $\theta_j$ is a $\T_2$-provably functional relation from $\psi^*_j$ to $\psi$. 
Firstly, by an instance of one of the disjunction axioms we have
$
 \T_2 \models \xi(\vec{x},\vec{y}) \wedge \psi^* (\vec{\bar{x'}},\vec{y}) \vdash \bigvee \xi_j(\vec{x},\vec{y}) \wedge \psi_j^* (\vec{\bar{x'}},\vec{y}) 
$
where the conjunct of the left is one of the disjuncts on the right. Thus by (\ref{H}) we get
$
\T_2 \models \xi(\vec{x},\vec{y}) \wedge \psi^* (\vec{\bar{x'}},\vec{y}) \vdash \psi 
$ 
and by an instance for one of the axioms for conjunction we also have
$
\T_2 \models \xi(\vec{x},\vec{y}) \wedge \psi^* (\vec{\bar{x'}},\vec{y}) \wedge \vec{\bar{x'}} = \vec{\bar{x}} \vdash \psi^* (\vec{\bar{x'}},\vec{y}) 
$
. From these last two sequents and $\wedge$-intro we then get
\[
\T_2 \models \xi(\vec{x},\vec{y}) \wedge \psi^* (\vec{\bar{x'}},\vec{y}) \wedge \vec{\bar{x'}} = \vec{\bar{x}} \vdash \psi^* (\vec{\bar{x'}},\vec{y}) \wedge  \psi (\vec{\bar{x}},\vec{x})
\]
To see that
\[
\T_2 \models \xi(\vec{x},\vec{y}) \wedge \psi^* (\vec{\bar{x'}},\vec{y}) \wedge \vec{\bar{x'}} = \vec{\bar{x}} \wedge \xi(\vec{x''},\vec{y}) \wedge \psi^* (\vec{\bar{x'}},\vec{y}) \wedge \vec{\bar{x'}} = \vec{\bar{x''}} \vdash \vec{\bar{x''}} = \vec{\bar{x}} \wedge \vec{x''} = \vec{x}
\]
we note that first conjunct on the right follows directly form the formula on the left as above. and the second conjunct on the right follows from Lemma \ref{codefunc}.
Finally, we need
 $\T_2 \models \psi^*(\vec{\bar{x'}},\vec{y}) \vdash \exists \vec{x} \exists \vec{\bar{x}} (\xi(\vec{x},\vec{y}) \wedge \psi^* (\vec{\bar{x'}},\vec{y}) \wedge \vec{\bar{x'}} = \vec{\bar{x}} )$.
In case $\psi$ has no variables of subsorts this follows immediately from Lemma \ref{codex} by noting that the right-hand side of the sequent is $\T_2$-provably equivalent to $\exists \vec{x} (\xi(\vec{x},\vec{y}) \wedge \psi^* (\vec{\bar{x}},\vec{y}))$. If on the other hand a variable of a subsort $S$ of a sort $T$ (via $i$ and $\phi$) is involved then we must rely on the explicit definition of $\psi^*$, which, in the case of subsorts, we have stipulated contains a conjunct $\phi(y)$ for some $y \colon T$ appearing among the $\vec{y}$. If $\phi \equiv \bot$ (i.e. if we have a variable of the empty subsort in $\psi^*$) then $\psi^*$ contains $\bot$ as one of its conjuncts and is therefore $\T_2$-provably equivalent to $\bot$ and the result follows trivially from EFQ. 
\end{proof}

\begin{prop}\label{comlemmhyp}
For every object $\lb \vec{y}.\psi \rb$ of $\C_{\T_2}$ there exists a covering family in $J_2(\lb \vec{y}.\psi \rb)$ generated by morphisms all of whose domains are in $\C_{\T_1}$.
\end{prop}
\begin{proof}
Let $\lb \vec{\bar{x}},\vec{x}. \psi \rb$ be an object of $\C_{\T_2}$ with $\vec{\bar{x}}$ variables of sorts in $\Sigma_1$ and $\vec{x}$ variables of sorts in $\Sigma_2$. By the second part of Lemma \ref{key}, we have that there are (finitely many) morphisms
$
[\theta_j] \colon \lb \vec{\bar{x}},\vec{x}. \psi^*_j \rb \rightarrow \lb \vec{\bar{x}},\vec{x}. \psi \rb
$
where each $\theta_j$ is of the form described in the proof of Lemma \ref{key}. Their images are given by the subobjects
$
[\exists \vec{y}_j \theta_j] \colon \lb \vec{\bar{x}}, \vec{x} . \exists \vec{y}_j \theta_j \rb \hookrightarrow \lb \vec{\bar{x}},\vec{x}. \psi \rb
$
and the union of all these subobjects is given by the following subobject
$
[\bigvee_j\exists \vec{y}_j \theta_j] \colon \lb \vec{\bar{x}}, \vec{x} . \bigvee_j \exists \vec{y}_j \theta_j \rb \hookrightarrow \lb \vec{\bar{x}},\vec{x}. \psi \rb
$.
But by Lemma \ref{key}, we have
$
\T_2 \models \bigvee_j\exists \vec{y}_j \theta_j \dashv \vdash \psi
$
which by \cite{Elephant} D1.4.4(iv) implies that $[\bigvee_j\exists \vec{y}_j \theta_j]$ is indeed the maximal subobject. But then this means that the family $[\theta_j]$ generates a $J_2$-cover. Since all $\psi_j^*$ are $\Sigma_1$-formulas, we are done.
\end{proof}

We are now ready to prove the desired result.

\begin{theo}\label{main}
Let $\T$ and $\Tp$ be T-Morita equivalent coherent theories. Then they are J-Morita equivalent.
\end{theo}
\begin{proof}
Clearly it suffices to show that for any theories $\T_1$ and $\T_2$ such that $\T_2$ is a Morita extension of $\T_1$ then $\T_2$ and $\T_1$ are J-Morita equivalent. 
By Lemmas \ref{fullfaith},\ref{indeqcoh} and Proposition \ref{comlemmhyp} we have that all the hypotheses of the Comparison Lemma are satisfied for the sites $(\C_{\T_1},J_1)$ and $(\C_{\T_2},J_2)$. This means that they give rise to equivalent categories of sheaves and hence (by \cite{Elephant} D3.1.9) $\T_1$ and $\T_2$ have equivalent classifying toposes. 
\end{proof}

\section{J-Morita to T-Morita}\label{jmortotmor}

We will now prove a converse to Theorem \ref{main}. 
Before we begin we should note that ``terminal object sorts'' and ``initial object sorts'' can be defined in terms of quotients and subsorts, as long as our signature contains at least one (non-empty) sort symbol, which we have been assuming throughout. So let $\T$ be a coherent theory over a signature $\Sigma$ and let $S$ be such a non-empty sort. Let $\phi (x,y) \equiv x=x \wedge y=y$ and let $1$ be the quotient sort defined from $S$ using $\phi$ with $\epsilon \colon S \rightarrow 1$ the associated projection. Clearly it follows from (\ref{eq:quot1}) and (\ref{eq:quot2}) that
$
\T \models \top \vdash \exists x \colon 1 (x=x)
$
and
$
\T \models \top \vdash_{x,x' \colon 1} x=x'
$
We can then define new function symbols $!_T \colon T \rightarrow 1$ for any other sort symbols $T \in \Sigma\textbf{-Sort}$ by
$
!_T (x) = y \dashv \vdash_{x \colon T, y \colon 1} x=x \wedge y=y 
$.  
Similarly we can define $0$ to be the subsort of $S$ defined using the formula $\bot$ in which case we get
$
\T \models x=x \vdash_{x \colon 0} \bot
$
. We can then define a function symbol $0_T \colon 0 \rightarrow T$ for any other sort symbol $T \in \Sigma\textbf{-Sort}$ by
$
0_T (x) = y \dashv \vdash_{x \colon 0, y \colon T} \bot
$
. Clearly any $\Tp$ that contains any and all of these extra definitions (over the suitably expanded signature $\Sigma' \supset \Sigma$) will be a Morita extension of $\T$. 



As in the previous section, if $\T_2$ (over $\Sigma_2$) is a Morita extension of $\T_1$ (over $\Sigma_1$) then for $S$ any sort symbol in $\Sigma_2$ we will reserve the notation $\sigma_S$ for the collection of sequents defining $S$ in terms of $\T_1$. With this in mind, the rest of this section is devoted to proving the following result:

\begin{theo}\label{main2}
Let $\T$ and $\Tp$ be J-Morita equivalent coherent theories. Then they are T-Morita equivalent.
\end{theo}

For any coherent category $\C$ we write $\T_\C$ for its internal theory as a coherent category, i.e. the collection of all sequents satisfied by $\C$ over the canonical language $\Sigma_\C$ obtained in the usual way by taking $\Sigma_\C \textbf{-Sort} = \text{ob} (\C)$ and $\Sigma_\C \textbf{-Fun} = \text{mor} (\C)$ with the obvious sorting. For the purposes of this paper we will not consider the \emph{extended} canonical language in which we also add a relation symbol for each monomorphism.
Now if $\P$ is the pretopos completion of $\C$, then $\T_{\mathcal{\P}}$ is obtained from $\T_\C$ by adding coproducts and then quotients for equivalence relations (\cite{MakkaiReyes}, 8.4 (A) and (B)). Inspection of this construction easily shows that each of these two steps are exactly Morita extensions in our sense. Therefore, for a coherent theory $\T$, there is a Morita span from $\T_{\C_\T}$ to $\T_{\P_\T}$.
It is also easy to see that if $\C$ and $\D$ be equivalent coherent categories then $\T_\C$ and $\T_\D$ are T-Morita equivalent (just take the internal theory of a common skeleton.) From all this we get:

\begin{prop}\label{tct}
Let two coherent theories $\T$ and $\Tp$ be J-Morita equivalent. Then the internal theories $\T_{\C_\T}$ and $\T_{\C_{\Tp}}$ of their syntactic categories are T-Morita equivalent.
\end{prop}

We now turn our attention to the more difficult problem of establishing that for any coherent theory $\T$, the internal theory $\T_{\C_\T}$ of its syntactic category $\C_\T$ is T-Morita equivalent to $\T$. 
We do so by explicitly constructing a Morita span between $\T$ and $\T_{\C_\T}$.
Before doing this we introduce some notation. 
Given a sort symbol $S \in \Sigma$ we let $\Sigma(S)$ be the set of $\Sigma$-formulas in a context consisting of a single variable of sort $S$ (which is assumed suitable for them.) Similarly, if $\vec{S}=\prod_{i=1}^n S_i \in \bar{\Sigma}\textbf{-Sort}$ for some $\bar{\Sigma} \supset \Sigma$ we let $\Sigma(\vec{S})$ be the set of $\Sigma$-formulas in the context $x_1 \colon S_1,...,x_n \colon S_n$ (which is assumed suitable for them.) 
For any $\phi \in \Sigma(\vec{S})$ we write
\[
\widetilde{\phi} (s) \equiv \exists x_1 \colon S_1 ... \exists x_n \colon S_n (\xi(s,\vec{x}) \wedge \phi(\vec{x}))
\]
where $\xi$ is a code for $s$. 
We will use the notation $\vec{S}_\phi$ to denote the subsort of $\vec{S}$ defined by $\widetilde{\phi}$ and we will denote by $i_\phi$ the associated function symbol (dropping the tildes). We will also adopt the following notational convention: if $\vec{S}$ is being considered as a product of sorts, we will denote variables of sort $\vec{S}$ by the lower-case version of the letter in question, i.e. in this case $s$. On the other hand, if we are considering $\vec{S}$ as a list of sorts of $\Sigma$ viewed in isolation, we will write $\vec{x} \colon \vec{S}$ to denote $x_1 \colon S_1,...,x_n \colon S_n$. 
The following proposition follows easily.

\begin{prop}\label{tildequiv}
Let $\T$ be a coherent theory over a signature $\Sigma$ and $\phi, \psi$ be $\Sigma$-formulas in a context $x_1 \colon S_1,..., x_n \colon S_n$ suitable for both of them. Let $\Tp$ be any Morita extension of $\T$ that contains the product sort $S_1 \times ... \times S_n$ to $\Sigma$. Then
\[
\T \models \phi \vdash_{\vec{x}\colon \vec{S}} \psi \text{ \emph{iff} } \Tp \models \widetilde{\phi} \vdash_{s \colon \vec{S}} \widetilde{\psi}
\]
\end{prop}

With all this in mind we can proceed to our constructions.

\begin{cons}[$\widehat{\T}$] 
We construct $\widehat{\T}$ from $\T$ in two steps: in the first step, we add arbitrary products and in the second step we add subsorts for any $\Sigma$-formula-in-context $\phi$ together with function symbols between these subsorts for all $\T$-provably functional relations between such formulas. More precisely we first expand the signature $\Sigma$ as follows
\begin{align*}
\Sigma_1 = \Sigma &\cup \lbrace \prod_{i=1}^n S_i \vert S_i \in \Sigma\textbf{-Sort}, n \in \mathbb{N} \rbrace \\ 
&\cup \lb \pi_i \colon \prod_{i=1}^n S_i \rightarrow S_i \vert i =1,...,n, \prod_{i=1}^n S_i \in (\Sigma_1 \setminus \Sigma)\textbf{-Sort}, n \in \mathbb{N} \rb
\end{align*}
and then define the following $\Sigma_1$-theory.
\[
\T_1 = \T \cup \lbrace \sigma_{\Pi_i S_i} \vert \prod_iS_i \in (\Sigma_1 \setminus \Sigma)\textbf{-Sort} \rb 
\]
We then we expand $\Sigma_1$ as follows
\begin{align*}
\widehat{\Sigma}=\Sigma_1 &\cup \lb S_\phi \vert S \in \Sigma_1\textbf{-Sort}, \phi \in \Sigma (S) \rb \cup \lb 1,0 \rb \\
&\cup \lb i_\phi \colon S_\phi \rightarrow S \vert S_\phi \in (\widehat{\Sigma} \setminus \Sigma_1)\textbf{-Sort} \rb \\
&\cup \lb !_S \colon S \rightarrow 1 \vert S \in \widehat{\Sigma}\textbf{-Sort} \rb \cup \lb 0_S \colon 0 \rightarrow S \vert S \in \widehat{\Sigma}\textbf{-Sort} \rb  \\
&\cup \lb \theta \colon S_\phi \rightarrow T_\psi \vert \text{$\theta$ is $\T$-provably functional from $\phi$ to $\psi$} \rb /{\sim} \\
&\cup \lb \pi_{\vec{S}} \colon \vec{S} \times \vec{T} \rightarrow \vec{S} \vert \vec{S}, \vec{S} \times \vec{T} \in \Sigma_1\textbf{-Sort} \rb
\end{align*}
where on the first line $1,0$ are ``terminal'' and ``initial'' sorts as explained in the beginning of this section and in the next-to-last line the symbol $\sim$ expresses $\T$-provable equivalence, i.e. we identify $\theta$ and $\theta'$ iff $\T \models \theta \dashv \vdash \theta'$ (over the appropriate context). Moreover we adopt the convention that $S_\top$ is the same as $S$, i.e. we do not add an extra sort symbol for the subsort defined by formula $\top$. Now for any such (equivalence class represented by a) functional relation $[\theta] \colon S_\phi \rightarrow T_\psi$ we add the following sequent, explicitly defining it:
\[
 [\theta](s)=t \dashv \vdash_{s, t} \exists \vec{x} \colon \vec{S} \exists \vec{y} \colon \vec{T} (\xi(\vec{x},i_\phi(s)) \wedge \xi (\vec{y},i_\psi(t)) \wedge \theta( \vec{x}, \vec{y})) \tag{$\tau_{\theta}$}
\]

Similarly, we add sequents defining the ``compound projection'' symbols $\pi_{\vec{S}}$ in the obvious way: given sorts $\vec{S} = S_1 \times ... \times S_n$ and $\vec{T} = T_1 \times ... \times T_m$ we have
\[
\pi_{\vec{S}} (t) = s \dashv \vdash_{t \colon \vec{S} \times \vec{T}, s \colon \vec{S}} \bigwedge_{i=1}^n \pi^{\vec{S}}_i (s) = \pi^{\vec{S} \times \vec{T}}_i = (t) \tag{$\gamma_{\vec{S},\vec{S}\times\vec{T}}$}
\]
where $\pi_i^{\vec{S}} \colon S_1 \times ... \times S_n \rightarrow S_i$ are the projections associated to the product of the sorts in $\vec{S}$ and $\pi_i^{\vec{S} \times \vec{T}} \colon T_1 \times ... \times T_m \rightarrow T_i$ are the projections associated to the product of the sorts in $\vec{S}$ together with $\vec{T}$.
Thus we get:
\begin{align*}
\widehat{\T} = \T &\cup \lb \sigma_{S_\phi} \vert S_\phi \in (\widehat{\Sigma}\setminus \Sigma_1)\textbf{-Sort}\rb \\
&\cup \lb \tau_\theta \vert [\theta] \in (\widehat{\Sigma} \setminus \Sigma)\textbf{-Fun} \rb \\
&\cup \lb \gamma_{\vec{S}, \vec{S} \times \vec{T}} \vert  \vec{S}, \vec{S} \times \vec{T} \in \Sigma_1\textbf{-Sort} \rb
\end{align*}
Clearly $\T_1$ is a Morita extension of $\T$ since we are only adding explicit definitions of new sort symbols. $\widehat{\T}$ is a Morita extension of $\T_1$ since we are only adding explicit definitions of new sort symbols (including a singleton sort) and new function symbols. Thus we have defined a Morita chain from $\T$ to $\widehat{\T}$.
\end{cons}


\begin{cons}[$\widetilde{\T}$] 
We are given the syntactic category $\C_\T$ of $\T$ viewed as a coherent category. Firstly, without loss of generality we pass to a skeleton of $\C_\T$ (described by what follows.)  
In particular we now choose one specific sort symbol for (the object part) of every limit in $\C_\T$ together with specific (formulas representing the associated) universal arrows.
The choices we make are the canonical ones, as laid out in D1 of \cite{Elephant}. 
Let $\Sigma_\T$ be the signature of the canonical language of $\C_\T$ (i.e. of the skeleton of our original $\C_\T$.) We take $\Sigma_\T$ to consist of symbols in $\widehat{\Sigma}$ (this is our choice of skeleton): we write $\lb x_1 \colon S_1,...,x_n \colon S_n. \phi \rb$ as $(S_1 \times...\times S_n)_\phi$ (or $\vec{S}_\phi$ for short) and we use the already available symbols in $\widehat{\Sigma}$ for the function symbols.
This also means that $\lb x \colon S. \top \rb$ will be written as $S$, $\lb []. \top \rb$ as $1$ and $\lb [].\bot \rb$ as $0$.
Since we are considering objects of $\C_\T$ up to renaming of variables all this involves no loss of information. 
 Also, for example for two objects $\vec{S}_\phi$ and $\vec{T}_\psi$ we will take their product to be given by $(\vec{S}\times \vec{T})_{\phi \wedge \psi}$. 
 So this rewriting clearly exhibits $\Sigma_\T$ as a subsignature of $\widehat{\Sigma}$. 
We now define the following extension of $\Sigma_\T$
\begin{align*}
\widetilde{\Sigma} = \Sigma_\T &\cup \lb f \colon \vec{S} \rightarrow T \vert f \in \Sigma\textbf{-Fun} \rb \cup \lb R \subset \vec{S} \vert R \in \Sigma\textbf{-Rel} \rb \\
&\cup \lb \pi_{i} \colon S_1 \times ... \times S_n \rightarrow  S_i \vert S_1\times ... \times S_n \in \Sigma_\T\textbf{-Sort}, i \in \lb 1,...,n \rb  \rb \\
&\cup \lb i_\phi \colon \vec{S}_\phi \rightarrow \vec{S} \vert \vec{S} \in \Sigma_{\T}\textbf{-Sort},  \phi \in \Sigma(\vec{S})\rb \\
&\cup \lb !_S  \colon S \rightarrow 1 \vert S \in \Sigma_\T\textbf{-Sort} \rb \cup \lb 0_S \colon 0 \rightarrow S \vert S \in \Sigma_\T \textbf{-Sort} \rb
\end{align*}
where function symbols $f$ are defined as having arity $S_1,...,S_n \rightarrow T$ rather than $S_1 \times ... \times S_n \rightarrow T$, and relation symbols $R$ are defined as $n$-ary predicates $R \subset S_1\times ... \times S_n$ rather than as unary predicates $R \subset \vec{S}$. 
Hence we define the following $\widetilde{\Sigma}$-theory $\widetilde{\T}$:
\begin{align*}
\widetilde{\T} = \T_{\C_\T} \cup \lb &\delta_{\vec{S},i} \equiv \pi_i (s) = x \dashv \vdash_{s \colon \vec{S}, x \colon S_i} [x_i = x_i](s) = x \vert   \rb \\
\cup \lb &\epsilon_{\vec{S}, \phi} \equiv i_\phi (s) = x \dashv \vdash_{s,x} [\phi](s) = x \vert \vec{S} \in \Sigma_{\T}\textbf{Sort},  \phi \in \Sigma(\vec{S}) \rb  \\  
\cup \lb &\zeta_S \equiv 0_S (s) = x \dashv \vdash_{s,x} \bot \vert S \in \Sigma_\T \textbf{-Sort} \rb \\ 
\cup  \lb &\iota_S \equiv \text{ }  !_S (s) = x \dashv \vdash_{s,x} s=s \wedge x=x \vert S \in \Sigma_\T \textbf{-Sort} \rb \\ 
\cup \lb &\mu_f \equiv   f(\vec{x})=y \dashv \vdash_{\vec{x} \colon \vec{S}, y \colon T}  \exists s \colon \vec{S} \exists t \colon T (\xi(s,\vec{x}) \wedge \xi(t,y) \wedge \\  &[f(\vec{x})=y] (s) = t  \vert f \in \Sigma\textbf{-Fun}\rb \\ 
\cup \lb &\nu_R \equiv R\vec{x} \dashv \vdash_{\vec{x} \colon \vec{S}} \exists r \colon \vec{S}_{R\vec{x}} \exists s \colon \vec{S} (\xi(s,\vec{x}) \wedge i_{R\vec{x}} (r) = s) \vert R \in \Sigma\textbf{-Rel} \rb \\ 
\end{align*}
where the codes used in sequents $\mu_f$ and $\nu_R$ are of course expressed in terms of the symbols defined by the sequents $\delta_{\vec{S},i}$ and $\epsilon_{\vec{S},\phi}$. Clearly $\widetilde{\T}$ is a Morita extension of $\T_{\C_\T}$ since it is obtained from the latter by the addition of explicit definitions for new function and relation symbols. 
\end{cons}


\begin{lemma}\label{tilde}
Let $\phi$ be any $\Sigma$-formula, viewed as a $\widetilde{\Sigma}$-formula. Then we have:
\[
\widetilde{\T} \models \widetilde{\phi} \dashv \vdash_{s \colon \vec{S}} \exists y \colon \vec{S}_\phi (i_\phi (y) = s)
\]

\end{lemma}
\begin{proof}
We proceed by induction on the complexity of $\phi$. If $\phi \equiv R\vec{x}$ for some relation symbol $R \in \Sigma\textbf{-Rel}$ of arity $\vec{S}$ then we have
\begin{align*}
\widetilde{\T} \models \widetilde{R\vec{x}} &\dashv \vdash_{s \colon \vec{S}} \exists \vec{x} \colon \vec{S} (\xi(s,\vec{x}) \wedge R\vec{x}) \\ 
&\dashv \vdash_{s \colon \vec{S}} \exists \vec{x} \colon \vec{S} (\xi(s,\vec{x})\wedge(\exists s' \colon \vec{S}\exists y \colon \vec{S}_{R\vec{x}} (\xi(s',\vec{x}) \wedge i_{R\vec{x}} (y) = s'))) \tag{$\nu_R$}\\
&\dashv \vdash_{s \colon \vec{S}} \exists y \colon \vec{S}_{R\vec{x}} (\exists \vec{x} \colon \vec{S}( \xi(s,\vec{x}) \wedge i_{R\vec{x}} (y) = s  ))  \tag{\ref{codefunc}} \\
&\dashv \vdash_{s \colon \vec{S}} \exists y \colon \vec{S}_{R\vec{x}} (i_{R\vec{x}} (y) = s ) \tag{\ref{codex}}
\end{align*}

If $\phi \equiv s=t$ for $\Sigma$-terms $s,t$ then it clearly suffices to prove the proposition only for simple atomic formulas, i.e. for $s \equiv f (\vec{x})$ for some $f \in \Sigma\textbf{-Sort}$ and for $t\equiv y$, a variable of the appropriate sort given the arity of $f$. The proof then follows exactly analogously as in the case of relation symbols.

Now suppose $\phi \equiv \phi_1 \wedge \phi_2$ such that the inductive hypothesis holds for $\phi_1$ and $\phi_2$. This means we have
\[
\widetilde{\T} \models \widetilde{\phi_1} \dashv \vdash_{s \colon \vec{S}} \exists y \colon \vec{S}_{\phi_1} (i_{\phi_1} (y) = s), \widetilde{\phi_2} \dashv \vdash_{s \colon \vec{S}} \exists y \colon \vec{S}_{\phi_2} (i_{\phi_2} (y) = s) \tag{IH} \label{IH1a} 
\]
Note that the following diagram commutes in $\C_\T$ and that each square (and therefore each rectangle) is a pullback:
\[
\xymatrix{
(\vec{S} \times \vec{T})_{\phi_1 \wedge \phi_2} \ar[r]^q \ar[d]_{p} &(\vec{S} \times \vec{T})_{\phi_2 \wedge \vec{x} = \vec{x}} \ar[d]_{p_{22}} \ar[r]^{p_{21}} &\vec{T}_{\phi_2} \ar[d]^{i_{\phi_2}} \\
(\vec{S} \times \vec{T})_{\phi_1 \wedge \vec{y} = \vec{y}} \ar[r]_{p_{12}} \ar[d]_{p_{11}} &\vec{S} \times \vec{T} \ar[r]^{\pi_{\vec{T}}} \ar[d]_{\pi_{\vec{S}}} &\vec{T} \\ 
\vec{S}_{\phi_1} \ar[r]^{i_{\phi_1}} &\vec{S}
}
\]
In the deduction below we will refer to the sequents satisfied by $\widetilde{\T}$ related to properties of this diagram by the obvious abbreviations. So when we write (comm) below we are invoking the relevant sequent expressing the commutativity of some part of the diagram -- similarly with (pull) for pullbacks and (inj) for some map being a mono.
With this in mind we have:
\begin{align*}
\widetilde{\T} \models \widetilde{\phi_1 \wedge \phi_2} (r) &\dashv \vdash_{r \colon \vec{S} \times \vec{T}} \exists s\colon \vec{S} \exists t \colon \vec{T} (\xi(r,s,t) \wedge \widetilde{\phi_1} \wedge \widetilde{\phi_2}) \\
&\dashv \vdash_r \exists s \exists t \exists z \colon \vec{S}_{\phi_1} \exists w \colon \vec{T}_{\phi_2} (\xi(r,s,t) \wedge \tag{\ref{IH1a}} \\
&i_{\phi_1} (z)=s \wedge i_{\phi_2} (w) = t) \\
&\dashv \vdash_r \exists z \exists w (i_{\phi_1} (z) = \pi_{\vec{S}} (r) \wedge i_{\phi_2} (w) = \pi_{\vec{T}} (r)  ) \tag{subst} \\ 
&\dashv \vdash_r \exists z \exists w \exists x \colon (\vec{S} \times \vec{T})_{\phi_1 \wedge \vec{y} = \vec{y}} \exists  y \colon (\vec{S} \times \vec{T})_{\phi_2 \wedge \vec{x} = \vec{x}} \tag{pull} \\
&(p_{11}(x) = z \wedge p_{12}(x) = r \wedge p_{21}(y) = w \wedge p_{22}(y) = r) \\
&\dashv \vdash_r  \exists z \exists w \exists x \exists y \exists v \colon (\vec{S} \times \vec{T})_{\phi_1 \wedge \phi_2} \tag{pull} \\ 
&(p_{11}(x) = z \wedge p_{21}(y) = w \wedge \\
&p(v) = x \wedge q(v)=y \wedge p_{12}(x) = r \wedge p_{22}(y) = r) \\
&\dashv \vdash_r \exists v (p_{12}(p(v)) = r \wedge p_{22}(q(v)) = r \tag{subst} \\
&\wedge \exists z \exists w (p_{11}(p(v)) = z \wedge p_{21}(q(v)) = w)) \tag{\ref{eq:frob}}\\ 
&\dashv \vdash_r \exists v (i_{\phi_1 \wedge \phi_2} (v) = r \tag{comm} \\ 
&\wedge \exists z \exists w (i_{\phi_1} (p_{11}(p(v))) = i_{\phi_1} (z) \wedge i_{\phi_2}(p_{21}(q(v))) = i_{\phi_2} (w))) \tag{inj} \\ 
&\dashv \vdash_r \exists v (i_{\phi_1 \wedge \phi_2} (v) = r \tag{comm} \\
&\wedge  \exists z \exists w (\pi_{\vec{S}}(p_{12}(p(v))) = i_{\phi_1} (z) \wedge \pi_{\vec{T}}(p_{22}(q(v))) = i_{\phi_2} (w))) \\
&\dashv \vdash_r \exists v (i_{\phi_1 \wedge \phi_2} (v) = r \tag{subst} \\
&\wedge  \exists z \exists w (\pi_{\vec{S}}(r) = i_{\phi_1} (z) \wedge \pi_{\vec{T}}(r) = i_{\phi_2} (w))) \tag{p}\\
&\dashv \vdash_r \exists v (i_{\phi_1 \wedge \phi_2} (v) = r) \tag{$*$}
\end{align*}
The first sequent is by  a simple unpacking of the definition of $\widetilde{\phi_1 \wedge \phi_2}$ where $\xi(r,s,t)$ stands for $\pi_{\vec{S}} (r) = s \wedge \pi_{\vec{T}} (r) = t$. The last step $(*)$ is accomplished by noting that the conjunct in the penultimate sequent labelled (p) is the same as the (RHS of the) sequent appearing at the third step. Since we know that everything that follows after the third step is logically equivalent to it, we may thus eliminate the conjunct in moving to $(*)$. 
%
The cases $\phi \equiv \phi_1 \vee \phi_2$ and $\psi \equiv \exists y \phi$ proceed by exactly analogous internal language arguments involving the sups and image factorizations of the relevant subobjects in $\C_\T$. We omit the details.

\end{proof}

\begin{prop}\label{hateqtil}
Let $\T$ be a coherent theory over a signature $\Sigma$. Then $\widehat{\T}$ is logically equivalent to $\widetilde{\T}$.
\end{prop}

\begin{proof}
We need to prove that each sequent in $\widehat{\T}$ is derivable from sequents in $\widetilde{\T}$ and vice versa.

$\widetilde{\T} \subset \widehat{\T}$: We take each of the axioms of $\widetilde{\T}$ and show that they are derivable in $\widehat{\T}$.

\textbf{Commutative Diagrams}: We have $\widetilde{\T} \models \top \vdash_{x \colon \vec{S}_\phi} [\alpha](x) = [\gamma] ([\theta] (x))$ for every commutative diagram
\[
\xymatrix{
&\vec{T}_{\psi} \ar[rd]^{[\gamma]} \\
\vec{S}_\phi \ar[rr]^{[\alpha]} \ar[ur]^{[\theta]} & &\vec{Q}_{\chi}
}
\]
in $\C_\T$. Such a diagram is commutative if and only if 
\[
\T \models \alpha(\vec{x},\vec{z}) \dashv \vdash_{\vec{x},\vec{z}} \exists \vec{y} \colon \vec{T} (\theta (\vec{x},\vec{y}) \wedge \gamma (\vec{y},\vec{z}))
\]
by the definition of the  syntactic category $\C_\T$. This means that $[\alpha]$ and $[\exists y \theta \wedge \gamma]$ are the same function symbol in $\widehat{\T}$ since by construction we consider these function symbols only up to $\T$-provable equivalence. Starting with $\tau_{\exists y(\gamma \wedge \theta)}$ we thus have the following derivation:
\begin{align*}
\widehat{\T} \models [\exists y(\theta \wedge \gamma)] (s) = q &\dashv \vdash_{s,q} \exists \vec{x} \exists \vec{z} (\xi(s,\vec{x},q,\vec{z}) \wedge \exists \vec{y} (\theta \wedge \gamma)) \\
&\dashv \vdash_{s,q} \exists \vec{x} \exists \vec{z} \exists \vec{y} (\xi(s,\vec{x},q,\vec{z}) \wedge (\theta \wedge \gamma)) \tag{\ref{eq:frob}}\\
&\dashv \vdash_{s,q} \exists \vec{x} \exists \vec{z} \exists \vec{y} \exists t ((\xi(s,\vec{x},t, \vec{y}) \wedge \theta) \wedge (\xi(q,\vec{z},t, \vec{y}) \wedge \gamma)) \tag{\ref{codex},\ref{codefunc}} \\
&\dashv \vdash_{s,q} \exists t ([\gamma](t) = q \wedge [\theta](s) = t) \tag{$\tau_\theta, \tau_\gamma$} \\
&\dashv \vdash_{s,q} [\gamma] ([\theta](s)) = q \tag{\ref{etran}}
\end{align*}
So we have
$
\widehat{\T} \models [\exists y(\theta \wedge \gamma)] (s) = q \dashv \vdash_{s,q} [\gamma] ([\theta](s)) = q 
$
and by substituting $[\gamma] ([\theta](s))$ for $q$ in the above sequent we get
$$
\widehat{\T} \models \top \vdash_s [\exists y(\theta \wedge \gamma)] (s) = [\gamma] ([\theta](s)) \equiv \top \vdash_s [\alpha] (s) = [\gamma] ([\theta](s))
$$
as required.


\textbf{Finite Limits}: Next we must deal with all the sequents in $\widetilde{\T}$ expressing that a diagram is a (finite) limit diagram in $\C_\T$. 
For terminal objects we know that an object $1$ in $\C_\T$ is terminal if and only if $\C_\T$ satisfies
$
\top \vdash \exists x \colon 1 (x=x)
$
and
$
\top \vdash_{x,x' \colon 1} x=x'
$
. But those are exactly the sequents that $1$ satisfies as a ``terminal object sort'' in $\widehat{\T}$. So $\widehat{\T}$ also satisfies these sequents.
For binary products, let $\vec{S}_\phi$ and $\vec{T}_\psi$ be arbitrary sorts in $\Sigma_\T$ (i.e. objects in $\C_\T$). Then we have that 
\[
\widetilde{\T} \models p_1 (r) = p_1 (r') \wedge p_2 (r) = p_2 (r') \vdash_{r,r' \colon (\vec{S}\times\vec{T})_{\phi \wedge \psi}} r=r'
\]
\[
\widetilde{\T} \models \top \vdash_{s\colon \vec{S}_\phi, t \colon \vec{T}_\psi} \exists z (p_1(z) = s \wedge p_2 (z) = t)
\]
with projections given by $p_1 \equiv [\phi \wedge \psi]$ and $p_2 \equiv [\phi \wedge \psi]$. In order to show that $\widehat{\T}$ also satisfies these sequents, first note that the following diagram commutes in $\C_\T$:
\[
\xymatrix{
\vec{S}_\phi \ar[d]^{i_\phi}  &(\vec{S} \times \vec{T})_{\phi \wedge \psi} \ar[d]^{i_{\phi \wedge \psi}} \ar[l]_{p_1} \ar[r]^{p_2} &\vec{T}_\psi \ar[d]^{i_\psi} \\
\vec{S}  &\vec{S} \times \vec{T} \ar[l]^{\pi_{\vec{S}}} \ar[r]_{\pi_{\vec{T}}} &\vec{T} 
}
\]
Now since $\vec{S}\times\vec{T}$ is a product sort in $\widehat{\T}$, we have
\[
\widehat{\T} \models \pi_{\vec{S}} (z) = \pi_{\vec{S}} (z') \wedge \pi_{\vec{T}} (z) = \pi_{\vec{T}} (z') \vdash_{z,z'} z=z'
\]
By substitution we thus get
\[
\widehat{\T} \models \pi_{\vec{S}} (i_{\phi \wedge \psi} (r)) = \pi_{\vec{S}} (i_{\phi \wedge \psi} (r')) \wedge \pi_{\vec{T}} (i_{\phi \wedge \psi} (r)) = \pi_{\vec{T}} (i_{\phi \wedge \psi} (r')) \vdash_{r,r'} i_{\phi \wedge \psi} (r) = i_{\phi \wedge \psi} (r')
\]
Let's call the above sequent $\sigma$. From the fact that the above diagram commutes we know that 
\[
\widetilde{\T} \models \top \vdash_r i_\phi(p_1(r)) = \pi_{\vec{S}}(i_{\phi \wedge \psi}(r)), \top \vdash_r i_\psi(p_2(r)) = \pi_{\vec{T}}(i_{\phi \wedge \psi}(r))
\]
and therefore from part (1) of the Theorem proven above about commutative diagrams we get 
\[
\widehat{\T} \models \top \vdash_r i_\phi(p_1(r)) = \pi_{\vec{S}}(i_{\phi \wedge \psi}(r)), \top \vdash_r i_\psi(p_2(r)) = \pi_{\vec{T}}(i_{\phi \wedge \psi}(r)) \tag{$*$} \label{star}
\]
From the two sequents above combined with $\sigma$ we get
\[
\widehat{\T} \models (i_\phi(p_1(r)) = i_\phi(p_1(r'))) \wedge (i_\psi(p_1(r)) = i_\psi(p_1(r'))) \vdash_{r,r'} i_{\phi \wedge \psi} (r) = i_{\phi \wedge \psi} (r') 
\]
But $i_\phi, i_\psi$ and $i_{\phi \wedge \psi}$ are such that the ``injectivity'' sequent (\ref{eq:sub2}) is satisfied in $\widehat{\T}$ -- and this gives us that 
\[
\widehat{\T} \models p_1 (r) = p_1 (r') \wedge p_2 (r) = p_2 (r') \vdash_{r,r' \colon (\vec{S}\times\vec{T})_{\phi \wedge \psi}} r=r'
\]
as required.

On the other hand, the fact that $\widehat{\T} \models (\ref{eq:prod1})$ for the relevant product gives us
\[
\widehat{\T} \models \top \vdash_{x,y} \exists r \colon \vec{S} \times \vec{T} (\pi_{\vec{S}} (r) = i_\phi (x) \wedge \pi_{\vec{T}}(r) = i_\psi (y)) \label{eq:seq2} \tag{$\dag$}
\]
By (\ref{eq:sub1}) we have that
\[
\widehat{\T} \models \widetilde{\phi \wedge \psi} \dashv \vdash_{r \colon \vec{S} \times \vec{T}} \exists z \colon (\vec{S} \times \vec{T})_{\phi \wedge \psi} (i_{\phi \wedge \psi} (z) = r)
\]
But by definition we know that
\[
\widetilde{\phi \wedge \psi}(r) \equiv \exists s \colon \vec{S} \exists t \colon \vec{T} (\xi(r,s,t) \wedge \tilde{\phi} (s) \wedge \tilde{\psi} (t))
\] 
and and by applying (\ref{eq:sub1}) to the RHS of the above for the conjuncts $\widetilde{\phi}$ and $\widetilde{\psi}$ we get that
\begin{align*}
\widehat{\T} \models \widetilde{\phi \wedge \psi}(r) &\dashv \vdash_r \exists s \colon \vec{S} \exists t \colon \vec{T} (\xi(r,s,t) \wedge \exists x (i_\phi(x) = s) \wedge \exists y (i_\psi(y)=t)) \\
&\dashv \vdash_r \exists s \exists t \exists x \exists y (\xi(r,s,t) \wedge i_\phi(x) = s \wedge i_\psi(y)=t) \tag{\ref{eq:frob}}
\end{align*}
Now the RHS of the above sequent is clearly implied by the RHS of (\ref{eq:seq2}), which means then that the RHS of (\ref{eq:seq2}) implies $\widetilde{\phi \wedge \psi}$, i.e. we can derive from (\ref{eq:seq2}) the following
\begin{align*}
\widehat{\T} \models  \top &\vdash_{x,y} \exists r \colon \vec{S} \times \vec{T} (\pi_{\vec{S}} (r) = i_\phi (x) \wedge \pi_{\vec{T}}(r) = i_\psi (y) \wedge \widetilde{\phi \wedge \psi}(r)) \\
&\vdash_{x,y} \exists r (\pi_{\vec{S}} (r) = i_\phi (x) \wedge \pi_{\vec{T}}(r) = i_\psi (y) \wedge \exists z (i_{\phi \wedge \psi} (z) = r)) \tag{\ref{eq:sub1}} \\ 
&\vdash_{x,y} \exists r \exists z (\pi_{\vec{S}} (r) = i_\phi (x) \wedge \pi_{\vec{T}}(r) = i_\psi (y) \wedge  i_{\phi \wedge \psi} (z) = r) \tag{\ref{eq:frob}}\\ 
&\vdash_{x,y} \exists z (\pi_{\vec{S}} (i_{\phi \wedge \psi} (z)) = i_\phi (x) \wedge \pi_{\vec{T}}(i_{\phi \wedge \psi} (z)) = i_\psi (y)) \tag{subst}\\
&\vdash_{x,y}  \exists z ( i_\phi (p_1 (z)) = i_\phi (x) \wedge i_\psi(p_2 (z)) = i_\psi (y)) \tag{\ref{star}}\\
&\vdash_{x,y} \exists z (p_1 (z) = x \wedge p_2 (z) = y) \tag{\ref{eq:sub2}}
\end{align*}
which gives us the desired result. 
 An exactly analogous (and even simpler) argument as the above, gives us the required result for sequent in $\widetilde{\T}$ expressing that a diagram is an equalizer. We omit the details.

\textbf{Sups and Images}: Analogous, and omitted. 
%
%

\textbf{Explicit Definitions}: This follows easily since each of the sequents $\mu_f, \nu_R, \delta_{\vec{S},i}, \epsilon_{\vec{S}, \phi}, \zeta_S$ and $\iota_S$ in $\widetilde{\T}$ can easily be derived as special instances of the sequent $\tau_\theta$ for particular choices of $\theta$. For example if $f \colon \vec{S} \rightarrow T$ is a function symbol 
we can use
$
\tau_{f(\vec{x})=y} \equiv [f(\vec{x})=y](s)=t \dashv \vdash_{s,t} \exists \vec{x} \exists y (\xi(s,\vec{x}) \wedge \xi(t,y) \wedge f(\vec{x}) =y)
$
to derive $\mu_f$ by adding and eliminating quantifiers in accordance with properties of product sorts.
The remaining cases follow similarly, taking care of the usual degenerate cases. We omit the details.


$\widehat{\T} \subset \widetilde{\T}$: Similarly, we take each of the axioms of $\widehat{\T}$ and show that they are derivable in $\widetilde{\T}$.

\textbf{Sequents in $\T$}: We have that $\widehat{\T} \models \phi \vdash \psi$ whenever $\T \models \phi \vdash \psi$ since $\widehat{\T}$ is an extension of $\T$. Now whenever $\T \models \phi \vdash \psi$ we know by (\cite{Elephant}, Lemma D1.4.4(iv)) that in $\C_\T$ we have a morphism
$
\vec{S}_\phi \hookrightarrow \vec{S}_\psi
$
and moreover that this morphism is given by $[\phi]$ and that it fits in the following commutative triangle
\[
\xymatrix{
& \vec{S} \\
\vec{S}_\phi \ar[ru]^{i_\phi} \ar[rr]^{[\phi]} & &\vec{S}_{\psi} \ar[lu]_{i_\psi}
}
\]
Since $\widetilde{\T}$ is an extension of $\T_{\C_\T}$ we have that
$
\widetilde{\T} \models \top \vdash_{y \colon \vec{S}_\phi} i_\phi (y) = i_\psi ([\phi] (y))
$
and from this and $\wedge$-introduction and substitution it easily follows that
$
\widetilde{\T} \models \exists y \colon \vec{S}_\phi (i_\phi (y) = s) \vdash_{s} \exists z \colon \vec{S}_\psi (i_\psi (z) = s) 
$.
By Lemma \ref{tilde} followed by 
Proposition \ref{tildequiv} we then get
$
\widetilde{\T} \models \phi \vdash_{\vec{x}} \psi
$
as required.



 \textbf{Sequents for subsorts}: 
We use a similar argument as in the case of products. Let $\vec{S}_\phi \in \widehat{\Sigma}\textbf{-Sort}$ be a subsort of $\vec{S} \in \Sigma_1\textbf{-Sort}$ with injection $i_\phi \colon \vec{S}_\phi \rightarrow \vec{S}$. Then $\widehat{\T} \models (\ref{eq:sub1}), (\ref{eq:sub2})$. Now in $\C_\T$ we have the following mono
$
[\phi] \colon \vec{S}_\phi \hookrightarrow \vec{S} 
$
. By sequents $\epsilon_{\vec{S}, \phi}$ we may write $[\phi]$ as $i_\phi$. Thus by soundness as above, we get that
$
\widetilde{\T} \models i_\phi (s) = i_\phi (s') \vdash s = s
$
which is of course exactly the sequent (\ref{eq:sub2}). Now from Lemma \ref{tildequiv} we get 
$
\widetilde{\T} \models \widetilde{\phi} \dashv \vdash_{s \colon \vec{S}} \exists y \colon \vec{S}_\phi (i_\phi(y) = s)
$
which (given our convention of suppressing the tilde in subscripts) is exactly (\ref{eq:sub1}).

 \textbf{Sequents for product sorts}: Analogous to subsorts, and omitted.


\textbf{Explicit definitions}: We have that $\widehat{\T} \models \tau_\theta$ for all $\theta$ such that $[\theta] \in \widehat{\Sigma} \setminus \Sigma\textbf{-Fun}$. We need to show that $\widetilde{\T}$ also satisfies these sequents. To do so, first recall that $\theta$ is a $\Sigma$-formula in context $\vec{S}\times\vec{T}$, i.e. $\theta \in \Sigma(\vec{S} \times \vec{T})$. Assume that it is a $\T$-provable functional relation between $\phi \in \Sigma (\vec{S})$ and $\psi \in \Sigma(\vec{T})$, i.e. a morphism $[\theta] \colon \vec{S}_\phi \rightarrow \vec{T}_\psi$ in $\C_\T$. Thus we have
$
\widetilde{\theta}(r) \equiv \exists \vec{x},\vec{y} \colon \vec{S} \times \vec{T} (\xi(r,\vec{x},\vec{y}) \wedge \theta(\vec{x},\vec{y}))
$
and by Lemma \ref{tilde} we get
$
\widetilde{\T} \models \widetilde{\theta} \dashv \vdash_{r \colon \vec{S} \times \vec{T}} \exists z \colon (\vec{S} \times \vec{T})_\theta (i_{\theta}(z)=r)	
$
and therefore
\[
\widetilde{\T} \models \theta(\vec{x},\vec{y}) \dashv \vdash_{\vec{x},\vec{y}} \exists z \colon (S \times T)_\theta (\xi(\vec{x},\pi_{\vec{S}}(i_\theta(z)))  \wedge \xi(\vec{y},\pi_{\vec{T}}(i_\theta(z))))) \label{eq:tild} \tag{$\dag$}
\]
Now note that the following diagram commutes in $\C_\T$:
\[
\xymatrix{
&\vec{S} &\vec{S} \times \vec{T} \ar[l]_{\pi_{\vec{S}}} \ar[r]^{\pi_{\vec{T}}} & \vec{T} \\
\vec{S}_\phi \ar@/_30pt/[rrrr]_{[\theta]} \ar[ru]^{i_\phi} &(\vec{S} \times \vec{T})_{\phi \wedge \vec{y}=\vec{y}} \ar[l]_{[\phi]} \ar[ru]^{i_{\phi \wedge \vec{y}=\vec{y}}} &(\vec{S}\times\vec{T})_\theta \ar[l]_{[\theta]_1} \ar[r]^{[\theta]_2} \ar[u]^{i_\theta} &(\vec{S}\times\vec{T})_{\psi \wedge \vec{x}=\vec{x}} \ar[lu]_{i_{\psi \wedge \vec{x}=\vec{x}}} \ar[r]^{[\psi]} &\vec{T}_\psi \ar[lu]_{i_\psi} \\
}
\]
where $[\theta]_1$ and $[\theta]_2$ are the unique maps guaranteed to exist by (\cite{Elephant}, D1.4.4.(iv)) since $\T \models \theta \vdash \phi$ and $\T \models \theta \vdash \psi$. From the bottom part of the diagram and the fact that $\theta$ is a functional relation we thus get
\[
\widetilde{\T} \models [\theta](s)=t \dashv \vdash_{s,t} \exists z \colon (\vec{S} \times \vec{T})_\theta (s=p_1(z) \wedge t=p_2(z))
\]
where $p_1 = [\phi] \circ [\theta]_1$ and $p_2 = [\psi] \circ [\theta]_2$. Now we have the following derivation:
\begin{align*}
\widetilde{\T} \models [\theta](s)=t &\dashv \vdash_{s,t} \exists z \colon (\vec{S} \times \vec{T})_\theta (s=p_1(z) \wedge t=p_2(z)) \\
&\dashv \vdash \exists z \colon (\vec{S} \times \vec{T})_\theta (i_\phi (s)= i_\phi (p_1(z)) \wedge i_\psi(t)=i_\psi(p_2(z))) \tag{\ref{eq:sub2}} \\
&\dashv \vdash \exists z \colon (\vec{S} \times \vec{T})_\theta (i_\phi (s)= \pi_{\vec{S}}(i_\theta(z)) \wedge i_\psi(t)=\pi_{\vec{T}}(i_\theta(z))) \tag{comm} \\
&\dashv \vdash \exists \vec{x} \colon \vec{S} \exists \vec{y} \colon \vec{T} \exists z \colon (S \times T)_\theta (\xi(\vec{x},i_\phi(s)) \wedge \xi(\vec{y},i_\psi(t)) \tag{\ref{codex}}\\ &\wedge \xi(\vec{x},\pi_{\vec{S}}(i_\theta(z)))  \wedge \xi(\vec{y},\pi_{\vec{T}}(i_\theta(z)))) \\ 
&\dashv \vdash \exists \vec{x} \colon \vec{S} \exists \vec{y} \colon \vec{T} (\xi(\vec{x},i_\phi(s)) \wedge \xi(\vec{y},i_\psi(t)) \tag{\ref{eq:frob}}\\  &\wedge (\exists z \colon (S \times T)_\theta (\xi(\vec{x},\pi_{\vec{S}}(i_\theta(z)))  \wedge \xi(\vec{y},\pi_{\vec{T}}(i_\theta(z)))))) \\  
&\dashv \vdash \exists \vec{x} \colon \vec{S} \exists \vec{y} \colon \vec{T} (\xi(\vec{x},i_\phi(s)) \wedge \xi(\vec{y},i_\psi(t)) \wedge \theta(\vec{x},\vec{y})) \tag{$\dag$}
\end{align*}
where by (comm) we indicate the fact that the relevant sequent is satisfied since the above diagram commutes in $\C_\T$. 


%


\end{proof}

\begin{cor}\label{ttct}
For any coherent theory $\T$, $\T$ is T-Morita equivalent to $\T_{\C_\T}$.
\end{cor}
\begin{proof}
By construction, there is a Morita chain from $\T$ to $\widehat{\T}$ and a Morita chain from $\T_{\C_\T}$ to $\widetilde{\T}$ and from Proposition \ref{hateqtil} we have that $\widehat{\T}$ and $\widetilde{\T}$ are logically equivalent.
\end{proof}


We are now finally ready to prove Theorem \ref{main2}.

\begin{proof}[Proof of Theorem \ref{main2}]
Let $\T$ and $\Tp$ be J-Morita equivalent coherent theories. By Corollary \ref{ttct} we know that $\T$ is T-Morita equivalent to $\T_{\C_\T}$ and  $\Tp$ is T-Morita equivalent to $\T_{\C_{\Tp}}$. But since $\T$ and $\Tp$ are J-Morita equivalent, by Proposition \ref{tct} we get that $\T_{\C_\T}$ and $\T_{\C_\Tp}$ are T-Morita equivalent. 
\end{proof}

Combining Theorem \ref{main} and Theorem \ref{main2} we have thus arrived at the desired characterization:

\begin{theo}\label{TiffJ}
Two coherent theories $\T$ and $\Tp$ are T-Morita equivalent if and only if they are J-Morita equivalent.
\end{theo}

\begin{cor}
Two coherent theories $\T$ and $\Tp$ are T-Morita equivalent if and only if $\T\textbf{-Mod} (\E) \simeq \Tp\textbf{-Mod}(\E)$ naturally for any Grothendieck topos $\E$.
\end{cor}
%

\begin{cor}
Two coherent theories $\T$ and $\Tp$ are T-Morita equivalent if and only if they have equivalent pretopos completions.
\end{cor}

\section{Generalizations}\label{gen}



For geometric theories, if we extend the definition of a Morita extension to infinitary coproducts in the obvious way, and extend our deductive system to include the relevant rules for infinitary distributivity of disjunction over conjunction, then all the proofs in Sections \ref{tmortojmor} and \ref{jmortotmor} still go through unchanged, except finite disjunctions will be replaced by arbitrarily large ones. The only issue that arises is the issue of defining categories of sheaves on large sites since $\C_\T$ is no longer necessarily small. 
This is only a minor difficulty: essentially the fix is contained in what is said in the proof of (\cite{Elephant}, Lemma D1.4.10(iv)).
As such we obtain the analogues of Theorems \ref{main} and \ref{main2} for geometric logic, if in their statement we take the notion of Morita extension and Morita equivalence in the expanded sense, in which we include infinitary coproducts. 

\begin{theo}\label{TiffJgeom}
Let $\T_1$ and $\T_2$ be geometric theories. Then they are J-Morita equivalent if and only if they are T-Morita equivalent.
\end{theo}

In the case of regular theories, we can define a \emph{regular Morita extension} (resp. \emph{regular T-Morita equivalence}) to be the same notion as described in Definition \ref{morex} but without coproduct sorts. Then we get:

\begin{theo}\label{regmain}
Let $\T$ and $\Tp$ be regular theories. Then they are J-Morita equivalent if and only if they are regular T-Morita equivalent.
\end{theo}
\begin{proof}[Proof Sketch]
For sufficiency, the same proof strategy as Section \ref{tmortojmor} works again here. 
For necessity, we may once again repeat the constructions of Section $\ref{jmortotmor}$ dropping any mention of coproduct sorts. One important difference however is how we apply the argument that leads up to Proposition \ref{tct} because we now we need the fact that the effectivization of a regular category $\C$ is equivalent to adding quotients of equivalent relations (but not coproducts) to $\T_\C$. This fact -- to our knowledge -- is nowhere directly recorded in the literature, although it is an easy consequence of many well-known constructions. The reader is referred to (\cite{Elephant}, 3.3.10) where the construction of $\textbf{Eff}(\C)$ given there in the setting of allegories can easily be seen to involve the free addition to $\T_\C$ of quotient sorts for equivalence relations.
\end{proof}

In the case of cartesian theories, we say that two cartesian theories $\T$ and $\Tp$ are \emph{(cartesian) J-Morita equivalent} if and only if their syntactic categories $\C_\T$ and $\C_{\Tp}$ are equivalent. The corresponding notion of \emph{cartesian T-Morita extension} (resp. \emph{cartesian T-Morita equivalence}) can be defined as in Definition \ref{Tmor} except we allow only product sorts, subsorts as well as singleton sorts (i.e. sorts $T$ satisfying the sequents
$
\top \vdash \exists x \colon T (\top) \label{sing1} 
$
and
$
\top \vdash_{x,y \colon T} x=y \label{sing2} 
$
) since we can no longer obtain singletons as quotient sorts.

\begin{theo}\label{cartmain}
Let $\T$ and $\Tp$ be cartesian theories. Then they are J-Morita equivalent if and only if they are cartesian T-Morita equivalent.
\end{theo}
\begin{proof}[Proof Sketch]
For sufficiency it suffices to note that if $\T_2$ (over $\Sigma_2$) is a cartesian Morita extension of $\T_1$ (over $\Sigma_1$) then the induced inclusion functor $i \colon C_{\T_1} \hookrightarrow \C_{\T_2}$ as in the proof of Lemma \ref{fullfaith} is now actually essentially surjective, and therefore an equivalence 
This can easily be seen to reduce to proving that each $\C_{\T_2}$-object of the form $\lb z \colon S \times T. \top \rb$ or $\lb s \colon S_\phi. \top \rb$ is isomorphic to $\lb x \colon S, y \colon T. \top \rb$ and $\lb x \colon S. \phi \rb$ respectively and both these facts follow straightforwardly from the relevant sequents for product sorts and subsorts.

Conversely, as in Theorem \ref{regmain}, we merely have to note that in the construction of $\widehat{\T}$ and $\widetilde{\T}$ only product and subsorts are used except in the case of the singleton sort which was there defined as a quotient sort and which will here be defined directly as a singleton sort as described above. The relevant parts of the proof of Proposition \ref{tildequiv} (those involving only commutative diagrams and finite limits) then follow as before.
\end{proof}


In the case of first-order theories 
there is no longer a ``good'' notion of a classifying topos and as a result no pre-existing notion of J-Morita equivalence. Even though it is possible to construct a topos that contains a generic model of a first-order theory $\T$ (cf. \cite{Elephant} D.3.1.18) geometric morphisms into this topos no longer correspond to models of $\T$. 
One can of course rectify the situation by adding suitable extra conditions (cf. e.g. \cite{JB}).
On the other hand, in \cite{Barrett} Barrett and Halvorson have shown that if two (classical) first-order theories $\T, \Tp$ are T-Morita equivalent then $\T \textbf{-Mod}_e (\textbf{Set}) \simeq \T' \textbf{-Mod}_e (\textbf{Set})$. 
%
From the well-known process of Morleyization (\cite{Elephant}, D1.5.13) we know that for any classical first-order theory $\T$ over a signature $\Sigma$ there is a coherent theory $\T_{m}$ over a signature $\Sigma_m \supset \Sigma$ such that
$
\T \textbf{-Mod}_e (\S) \simeq \T_m \textbf{-Mod} (\S)
$
where $\S$ is any Boolean coherent category. This suggests the following:

\begin{theo}\label{morlTiffJ}
Let $\T$ and $\Tp$ be be first-order theories. Then they are T-Morita equivalent if and only if their Morleyizations $\T_m$ and $\Tp_m$ are J-Morita equivalent as coherent theories.
\end{theo}

Before proving the theorem we require some preliminary results.



\begin{lemma}\label{Mequiv}
Let $\T$ be a first-order theory over classical logic and $\T_m$ its Morleyization regarded as a first-order theory. Then
$
\T_m \models \phi \dashv \vdash_{\vec{x}} C_\phi 
$
, where $\vec{x}$ is the canonical context for $\phi$.
\end{lemma}
\begin{proof}
This follows by a straightforward induction on the complexity of $\phi$. Before proceeding note also that
\[
\T_m \models D_\phi \dashv \vdash_{\vec{x}} \neg C_\phi  \label{eq:neg} \tag{neg}
\] 
since by construction we have
$\T_m \models \top \vdash C_\phi \vee D_\phi,  C_\phi \wedge D_\phi \vdash \bot$
and we are over classical logic.
Now if $\phi$ is atomic (including $\top$ or $\bot$) then the result holds by construction. For the inductive step, 
let us do only the case universal quantification since the rest follow similarly:
\begin{align*}
\T_m \models C_{\forall x \phi} &\dashv \vdash \neg D_{\forall x \phi} \tag{\ref{eq:neg}} \\
&\dashv \vdash \neg \exists x D_\phi \tag{by construction} \\
&\dashv \vdash \neg \exists x \neg \phi \tag{(\ref{eq:neg}),inductive hypothesis} \\
&\dashv \vdash \forall x \phi \tag{tautology}
\end{align*}
\end{proof}

\begin{prop}\label{fomor}
Any first order theory $\T$ is T-Morita equivalent to its Morleyization $\T_m$ when the latter is regarded as a first-order theory.
\end{prop}
\begin{proof}
We construct an extension ($\T_1$, $\Sigma_1$) of $\T$ as follows: we let $\Sigma_1 =\Sigma_m$, i.e. for every first-order formula $\phi$ over $\Sigma$ we add two relation symbols $D_\phi$ and $C_\phi$ of arity the same as the canonical context of $\phi$. We then define $\T_1$ as
$
\T_1 = \T \cup \lb C_\phi \dashv \vdash_{\vec{x}} \phi \vert \phi \in \Sigma\textbf{-foForm} \rb 
\cup \lb D_\phi \dashv \vdash_{\vec{x}} \neg \phi \vert \phi \in \Sigma\textbf{-foForm} \rb
$
where $\Sigma\textbf{-foForm}$ is the set of first-order $\Sigma$-formulas. Clearly $\T_1$ is a Morita extension (indeed a definitional extension) of $\T$ and it is over the same signature as $\T_m$. We claim that $\T_1$ and $\T_m$ are logically equivalent. 

$\T_1 \subset \T_m$: $\T_1$ consists of the axioms of $\T$ together with the explicit definitions of the new relation symbols as described above. So if $\phi \vdash_{\vec{x}} \psi$ is an axiom of $\T$, we know that $\T_m \models C_\phi \vdash_{\vec{x}} C_\psi$ and so by Lemma \ref{Mequiv} and substitution of equivalents we get that $\T_m \models \phi \vdash_{\vec{x}} \psi$. On the other hand, for every explicit definition $C_\phi \dashv \vdash \phi$ or $D_\phi \dashv \vdash \neg \phi$ in $\T_1$ we get that $\T_m$ satisfies them directly from Lemma \ref{Mequiv} and (\ref{eq:neg}) (as it appears in the proof of \ref{Mequiv}.)

$\T_m \subset \T_1$: If $C_\phi \vdash_{\vec{x}} D_\psi$ is an axiom of $\T_m$ then $\T \models \phi \vdash_{\vec{x}} \psi$ and therefore by the sequents explicitly defining $C_\phi$ and $C_\psi$ in $\T_1$ we get $\T_1 \models C_\phi \vdash_{\vec{x}} C_\psi$. 
Now we also have $\T_m \models \top \vdash_{\vec{x}} C_\phi \vee D_\phi$ and $\T_m \models C_\phi \wedge D_\phi \vdash_{\vec{x}} \bot$ for every first-order $\Sigma$-formula $\phi$. Since we are over classical logic we have:
\begin{align*}
\T_1 \models \top &\vdash_{\vec{x}} \phi \vee \neg \phi \tag{tautology} \\
&\vdash_{\vec{x}} C_\phi \vee D_\phi \tag{explicit definition axioms}
\end{align*}
We obtain
$
\T_1 \models C_\phi \wedge D_\phi \vdash_{\vec{x}}  \bot
$
exactly analogously. Finally, each of the sequents defining the relation symbols in $\T_m$ is also derivable in $\T_1$ immediately since the latter contains essentially exactly the same definitions. Let's do the universal quantifier as an illustration: we have that 
$
\T_m \models D_{\forall x\phi} \dashv \vdash_{\vec{y}} \exists x D_\phi
$.
To see that $\T_1$ also satisfies this sequent we can argue as follows:
\begin{align*}
\T_1 \models D_{\forall x \phi} &\dashv \vdash_{\vec{y}} \neg \forall x \phi \tag{axiom}\\
&\dashv \vdash_{\vec{y}}  \exists x \neg \phi \tag{tautology} \\
&\dashv \vdash_{\vec{y}} \exists x D_\phi \tag{axiom}
\end{align*}
\end{proof}

We are now ready to prove Theorem \ref{morlTiffJ}.

\begin{proof}(of Theorem \ref{morlTiffJ})
Let $\T_1$ (over $\Sigma_1$) be a first-order theory and $\T_2$ (over $\Sigma_2$) a Morita extension of $\T_1$. Let $\T_{1m}$ be the Morleyization of $\T_1$. Define $\T_{2m}$ over $\Sigma_2$ as follows. If $f$ is a function symbol in $\Sigma_2 \setminus \Sigma_1$\textbf{-Fun} explicitly defined by some $\Sigma_1$-formula $\phi$ as
$
\T_2 \models \top \vdash \forall \vec{x} \forall y (f(\vec{x})=y \leftrightarrow \phi(x,y))
$
then $\T_2$ contains the sequent
$
f(x)=y \dashv \vdash_{\vec{x},y} C_{\phi}(\vec{x},y)
$
. Similarly if $R$ is a relation symbol in $\Sigma_2 \setminus \Sigma_1$\textbf{-Rel} explicitly defined by some $\Sigma_1$-formula $\rho$ as 
$
\T_2 \models \top \vdash \forall \vec{x} (R\vec{x} \leftrightarrow \rho(\vec{x}))
$
then $\T_2$ contains the sequent
$
R \vec{x} \dashv \vdash_{\vec{x}} C_{\rho} (\vec{x})
$
. If $\T_2$ contains any new sort symbols then $\T_{2m}$ also contains the sequents attached to these new sort symbols -- since all these sequents are coherent, this poses no problem. Therefore $\T_{2m}$ is a Morita extension of $\T_{1m}$ and in fact it is not hard to see that $\T_{2m}$ is in fact exactly the Morleyization of $\T_2$ as usually constructed. Thus we may say that ``Morleyizing preserves Morita extensions.''  The only situation in which it is not immediate to verify this fact is if $\T_2$ contains subsorts or quotient sorts defined via first-order $\Sigma_1$-formulas. Take the case of subsorts as an illustration. Suppose that $S_\phi$ is a subsort in $\T_2$ defined via a $\Sigma_1$-formula $\phi$. This means that
$
\T_2 \models \phi(x) \dashv \vdash_{x \colon T} \exists y \colon S (i(y)=x)
$
and
$
\T_2 \models i(x)=i(y) \vdash_{x,y\colon S} x=y  
$
Therefore the Morleyization of $\T_2$ contains a sort symbol $S_\phi$ together with the sequents
$
C_\phi(x) \dashv \vdash_{x \colon T} \exists y \colon S (i(y)=x)
$
and
$
i(x)=i(y) \vdash_{x,y\colon S} x=y  
$
. But these last two sequents exactly define $S_\phi$ as a subsort of $S$ via the formula $C_\phi$. 
Since logically equivalent first-order theories have logically equivalent Morleyizations, this means that the Morita span from $\T$ to $\Tp$ induces a Morita span between $\T_m$ and $\Tp_m$. By Theorem \ref{TiffJ} this means that $\T_m$ and $\Tp_m$ are J-Morita equivalent.

Conversely, suppose that the Morleyizations of $\T_m$ and $\Tp_m$ of two first-order theories are J-Morita equivalent. By Theorem \ref{TiffJ} they are T-Morita equivalent. Clearly if two coherent theories are T-Morita equivalent as coherent theories then they are T-Morita equivalent as first-order theories. Thus $\T_m$ and $\Tp_m$ are T-Morita equivalent (as first-order theories.) But by Proposition \ref{fomor} we know that $\T$ and $\Tp$ are respectively T-Morita equivalent to $\T_m$ and $\Tp_m$, so we are done.
\end{proof}

\section{Applications and Implications}\label{appimp}

One of the main mathematical upshots of our result is that it gives a new way of generating Morita-equivalences between theories. As an indication of the possible mathematical rewards of this consider some recent remarks of L. Lafforgue \cite{Lafforgue} on the Langlands correspondence based on O. Caramello's work as best summarized in \cite{Caram}. 
Our result could also be another way to judge the viability of such an approach, on whose prospects we remain neutral. We intend to explore some tentative results in this direction in future work.



Another consequence of our result is that topos-theoretic invariants automatically become invariants of theories up to T-Morita equivalence. 
This is a kind of inversion of the philosophy expounded in \cite{Caram}. 
More precisely, given a topos-theoretic invariant $P$ 
(a property $P$ of a topos $\mathcal{E}$ invariant under geometric equivalences)
we automatically get a property $P'$ of a geometric theory $\T$ that is invariant under T-Morita equivalence. This is interesting in (at least) two ways. Firstly, it gives us a syntactic way of determining whether a property of a theory $\T$ corresponds to a topos-theoretic invariant (of the classifying topos of $\T$), by seeing whether or not it is invariant under T-Morita equivalence. 

\begin{exam}
For a trivial example, consider the property of being single-sorted. There are clearly examples of single-sorted (geometric) theories that are T-Morita equivalent to many-sorted ones. Automatically this means that there cannot be a topos-theoretic invariant property expressing the fact that the geometric theory corresponding to a particular topos is single-sorted.
\end{exam}

Secondly, it allows us to automatically conclude that existing topos-theoretic invariants correspond to T-Morita invariant properties of theories

\begin{exam}
It is known (cf. \cite{Caram2012}) that a complete geometric theory is countably categorical if and only if its classifying topos is atomic. Since being atomic is a topos-theoretic invariant, from Theorem \ref{TiffJgeom} this means that countable categoricity is invariant under T-Morita equivalence. 
\end{exam}


Finally T-Morita equivalence gives us a more precise way of understanding what is done by the following functors:
\[
\T \mapsto \C_\T \colon \textbf{CohTheo}  \rightleftarrows \textbf{PreTop} \colon \C \mapsto \T_\C
\]
In particular, taking a coherent theory to the pretopos completion of its syntactic category and then taking the internal theory of that pretopos (as a coherent category) gives us back a T-Morita equivalent theory. Since T-Morita equivalences are essentially spans of Morita extensions one can raise the question about whether \textbf{PreTop} can be characterized by a localization procedure on \textbf{CohTheo}. This would be a logical version of the localization constructions in \cite{Moerdijk1,Moerdijk2}. We plan to explore this idea in future work.

\subsection*{Acknowledgments} The author would like to thank Thomas Barrett, Hans Halvorson, Neil Dewar and Dan Dore for many helpful remarks during the preparation of this work.

\end{document}